\documentclass[reqno,12pt]{amsart}

\usepackage[toc]{appendix}
\usepackage[utf8]{inputenc}
\usepackage[T1]{fontenc}
\usepackage{microtype} 
\usepackage{cite}
\usepackage{graphicx}
\usepackage{hyperref}
\usepackage{mathrsfs} 
\usepackage{subcaption}
\usepackage{amsfonts,amsmath,amsthm,amssymb}
\usepackage[a4paper,bindingoffset=0.5cm,left=2cm,right=2cm,top=2.5cm,bottom=2cm,footskip=.8cm]{geometry}
\usepackage{pgfplots}

\usepackage{tikz}
\usetikzlibrary{automata,positioning}
\usetikzlibrary{decorations.pathreplacing,decorations.markings,shapes.geometric}
\usetikzlibrary{shapes,arrows}
\usetikzlibrary{backgrounds,calc,positioning}

\theoremstyle{definition}

\theoremstyle{plain}
\newtheorem{theorem}{Theorem}[section]

\newtheorem{remark}{Remark}[section]

\newtheorem{assumption}{Assumption}

\newcommand{\vb}{\vspace{3mm}}

\begin{document}

\allowdisplaybreaks

\title[A transient Cram\'er-Lundberg model]{{\small 
    A transient Cram\'er-Lundberg model\\ with applications to credit risk}}

\author{By Guusje Delsing and Michel Mandjes}

\begin{abstract} This paper considers a variant of the classical Cram\'er-Lundberg model that is particularly appropriate in the credit context, with the distinguishing feature that it corresponds to a finite number of obligors. The focus is on computing the ruin probability, i.e., the probability that the initial reserve, increased by the interest received from the obligors and decreased by the losses due to defaults, drops below zero. Besides an exact analysis (in terms of transforms) of this ruin probability, also an asymptotic analysis is performed, including an efficient importance-sampling based simulation approach. 

\noindent The base model is extended in multiple dimensions: (i)~we consider a model in which there may, in addition, be losses that do not correspond to defaults, (ii)~then we analyze a model in which the individual obligors are coupled through a regime-switching mechanism, (iii)~then we extend the model such that between the losses the reserve process behaves as a Brownian motion rather than a deterministic drift, and (iv)~we finally consider a set-up with multiple groups of statistically identical obligors. 
\vb


\vb

\noindent
{\sc AMS Subject Classification (MSC2010).} 
Primary: 
60G51 

\vb

\noindent
{\sc Affiliations.} 
{\it Guusje Delsing} is with Korteweg-de Vries Institute for Mathematics, University of Amsterdam, Science Park 904, 1098 XH Amsterdam, the Netherlands, and with Rabobank, Croeselaan 18, 3521 CB Utrecht, the Netherlands. Email: \url{g.a.delsing@uva.nl}.

\noindent
{\it Michel Mandjes} is with Korteweg-de Vries Institute for Mathematics, University of Amsterdam, Science Park 904, 1098 XH Amsterdam, the Netherlands. He is also with E{\sc urandom}, Eindhoven University of Technology, Eindhoven, the Netherlands, and Amsterdam Business School, Faculty of Economics and Business, University of Amsterdam, Amsterdam, the Netherlands. His research is partly funded by NWO Gravitation project N{\sc etworks}, grant number 024.002.003. Email: \url{m.r.h.mandjes@uva.nl}. Version: \today.

\noindent
\end{abstract}
\maketitle

\newpage

\section{Introduction}
In insurance and risk, a pivotal role is played by the classical {\it Cram\'er-Lundberg model} (also known as the {\it compound Poisson model}). In this model independent and identically distributed claims arrive according to a Poisson process, whereas premiums are earned at a constant rate. This means that if the initial reserve is given by $u\geqslant 0$, then the reserve level at time $t\geqslant 0$ is given by 
\begin{equation}X_t:= u + rt - \sum_{i=1}^{N_t} L_i,\label{CP}\end{equation}
with $r>0$ the premium rate, $(N_t)_{t\geqslant 0}$ a Poisson process with intensity $\lambda>0$, and $(L_i)_{i\in{\mathbb N}}$ a sequence of i.i.d.\ random variables. The key quantity of interest is the (finite-horizon) ruin probability
${\mathbb P}(\exists s\in[0,t]: X_s<0)$ and its infinite-horizon counterpart ${\mathbb P}(\exists s\geqslant 0: X_s<0)$. A broad set of techniques has been developed to analyze this quantity, for the Cram\'er-Lundberg model itself as well as for more advanced variants; we refer to \cite{AsmussenAlbrecher} for an exhaustive overview. With the random variable $L$ denoting a generic claim, often the {\it net profit condition} ${\mathbb E}\,(X_t-X_0)= rt - {\mathbb E}N_t\,{\mathbb E}L>0$ is imposed. Under this condition, which effectively means that $r> \lambda\,{\mathbb E}L$, it is guaranteed that ruin is rare. A practically relevant objective is to select the initial reserve $u$ such that the (finite or infinite-horizon) ruin probability is below some threshold $\varepsilon.$

Essentially the same modeling framework can be applied in the context of credit as well. Then the claim arrival process describes the default epochs, the premiums correspond to the interest received from the obligors, and the claims are the corresponding losses. One may wonder, however, whether in this setting the assumption of Poisson arrivals is any realistic: {whereas in the insurance context the number of claims issued can in principle exceed any bound, it is obvious that in the credit context the number of defaults cannot exceed the number of obligors. More concretely, as soon as an obligor goes into default, it effectively leaves the system.} Motivated by this observation, we study in this paper the ruin probability in a transient variant of the Cram\'er-Lundberg model. We do so by defining for each obligor a random variable (e.g.\ exponentially distributed) corresponding to the time-to-default, where after the default the obligor can neither cause any new default nor generates any interest anymore. 

\vb

{\it Model.} We proceed by providing a more formal description of our transient variant of the classical Cram\'er-Lundberg model. Here we state the main model, which we will generalize in various directions later in the paper.

We consider a setting in which there are initially $n\in{\mathbb N}$ obligors, each of which goes into default after some random amount of time. The corresponding $n$ times-to-default are assumed to be i.i.d.\ non-negative random variables, characterized by the density $f(\cdot)$. {In the credit context, risk is quantified over a finite time horizon justifying the use of a model in which clients eventually all go into default.} Let the loss-at-default, per obligor, be distributed as a non-negative random variable $L$, and let these losses be i.i.d., each with Laplace transform $\ell(\cdot)$. It is natural to assume that the income per unit of time is proportional to the number of obligors that have not gone into default yet. In other words, the surplus process increases at a rate $ri$ per unit of time, for some $r>0$, when there are $i$ obligors that have not defaulted yet, for $i\in\{0,\ldots,n\}.$ 
The company has an initial reserve level $u>0$. Because of the similarity with insurance and risk models, throughout this paper we sometimes refer to losses as claims. 

The primary objective of this paper is to evaluate $p_n(u,t)$, defined as the ruin probability of the company before time $t$, given there are $n$ obligors at time $0$ and that the initial reserve is $u$. Being able to compute $p_n(u,t)$, one can pick $u$ such that this ruin probability remains below an acceptable level $\varepsilon>0$. In addition, when a new obligor wishes to get a loan, knowledge of $p_n(u,t)$ allows one to decide if (and, if yes, by how much) the initial level should be adjusted.

\vb

{\it Contributions.} For the main model, we provide a procedure by which, for any $n$, the double transform (in space and time, that is) of $p_n(u,t)$ can be determined. More specifically, we develop a recursive relation by which these transforms can be determined. While this means that one can evaluate the finite-horizon ruin probability $p_n(u,t)$ by numerical inversion, we in addition point out how to efficiently estimate this rare-event probability relying on importance sampling simulations; the procedure proposed has provable optimality properties. In addition we provide the logarithmic asymptotics of $p_n(nu,t)$ as $n$ grows large (i.e., in this setting the initial reserve $u$ is scaled by the number of obligors $n$). 

Besides the base model, four generalizations are dealt with in this paper.
One could argue that the assumption of the times-to-default being independent is not realistic, as in reality defaults tend to cluster. To resolve this issue, in one of the generalizations we allow a {\it regime switching} mechanism (also frequently referred to as {\it Markov modulation}) that induces dependence between the obligors. The regime could be thought of as the `state of the economy', wherein every state of the economy the dynamics of the reserve level are described by a specific Cram\'er-Lundberg model. 
In a second generalization, we consider a model in which some loss events correspond to defaults (reducing the number of obligors by one) while others do not (leaving the number of obligors unchanged). Another unrealistic feature of the main model is that the obligors are homogeneous: their times-to-default (losses, respectively) stem from the same distribution. To remedy this, we
also analyze a model variant corresponding to heterogenous obligors: there are multiple groups, each of them consisting of statistically identical obligors. {This extension offers an important additional flexibility as one can cluster obligors based on the loss distribution, which is often deterministic in the credit context, and consider classes of obligors that do not go into default or have a class-specific income rate.}
A last extension that we discuss in this paper concerns a model in which between loss events the reserve level behaves as a Brownian motion (rather than as a deterministic drift).

\vb

{\it Related literature.} 
Starting from the pioneering papers by Cram\'er \cite{C1} and Lundberg \cite{L1,L2}, focusing on the classical compound Poisson model \eqref{CP},
a broad range of risk models has been analyzed. Without attempting to provide a complete overview, we proceed by discussing a few important branches; we refer to \cite{MIK,KGS, TEU} for general accounts of risk theory.
In the first place, the assumption of the cumulative claim process being of compound Poisson type has been lifted, thus allowing a compound Poisson claim process perturbed by a diffusion \cite{G2, G1}, and even a (spectrally one-sided) L\'evy claim process; see e.g.\ \cite[Ch.\ X and XI]{AsmussenAlbrecher} and \cite{DM,KYP}. In addition, some models incorporate returns on investment, while in other models the dynamics of the reserve process are level-dependent; see e.g.\ \cite[Ch.\ VIII]{AsmussenAlbrecher} and \cite{AC, BM}. Finally, there is a substantial body of papers exploring the effect of specific dependence structures; see e.g.\ \cite{CKM} and, for an overview, \cite[Ch.\ XIII]{AsmussenAlbrecher}. More specifically, the effect of parameter uncertainty can be analyzed through the resampling model recently proposed in \cite{CDMR}.

\vb

{\it Organization.} Section \ref{S3} provides an explicit analysis, in terms of transforms, for the base model introduced above. A large deviations analysis of the tail probability is presented in Section \ref{asy}, together with an importance-sampling based simulation approach and a uniform upper bound. The four extensions of the base model are presented by Section \ref{ext}. {The final section contains a series of numerical experiments.}

\section{Exact analysis}\label{S3}
In this section we analyze the base model that was described in the introduction. We start by defining the key quantities of this base model, pertaining to the case that each of the obligors has a time-to-default that is exponentially distributed.
We then present our analysis yielding a recursion for the double transform of the ruin probability. 

\subsection{Notation and preliminaries}
Per obligor the rate of going into default is $\lambda>0$. This means that if there are still $i$ obligors left (i.e., being not in default), the time till the next default is exponentially distributed with mean $(\lambda i)^{-1}$.

Recall that $p_n(u,t)$ is the probability of ruin before time $t$, starting with $n$ obligors at time $0$, given the initial reserve level is $u$.
In our approach we (uniquely) characterize $p_n(u,t)$ through its double transform
\[\psi_n(\gamma) := \int_0^\infty e^{-\gamma u} \int_0^\infty \vartheta e^{-\vartheta t} p_n(u,t)\,{\rm d}t\,{\rm d}u
=\int_0^\infty e^{-\gamma u}  p_n(u)\,{\rm d}u,\]
where $p_n(u)$ can be interpreted as the probability of ruin before an exponentially distributed clock with mean $\vartheta^{-1}$ (which is sampled independently from anything else). The case of $t=\infty$ corresponds with $\vartheta\downarrow 0$.
The main result of this section is an expression (recursive in $n$) for $\psi_n(\gamma)$: we express $\psi_n(\cdot)$ in terms of $\psi_{n-1}(\cdot)$. Observe that we can equivalently write $p_n(u)$ as ${\mathbb P}(Z_n\geqslant u)$, where $Z_n$ is the maximum of the {\it net} cumulative {loss} process (the net cumulative {claim} process, in the insurance context) over the above-mentioned exponentially distributed amount of time (with mean $\vartheta^{-1}$, that is). 

In practical settings, one typically has that $r>-\lambda\ell'(0)=\lambda\,{\mathbb E}L$, so that at any point in time ruin is rare, in the sense that the expected reserve increases as a function of time; to this end, realize that when there are $i\in\{0,\ldots,n\}$ obligors left, the `local drift' of the reserve process is $ri + \lambda i\,\ell'(0)>0.$

\subsection{Analysis}\label{subsec_analysis}
In this subsection we present a recursive scheme to evaluate $\psi_n(\gamma)$.
The main idea is to condition on the first event, being either the first default (which happens after an exponentially distributed time with mean $(\lambda n)^{-1})$ or the expiration of the exponential clock (which happens after an exponentially distributed time with mean $\vartheta^{-1})$. If the former event happens to apply first, then we can still reach ruin, but now with $n-1$ obligors and an adapted initial reserve. If the latter events occurs first, then we won't be facing ruin before the exponential clock expires.
These ideas can be translated into mathematical terms as
\begin{equation}
\label{RECU}p_n(u) = \int_0^\infty \lambda n\, e^{- (\lambda n+\vartheta) t} \,{\mathbb P}(Z_{n-1}+L\geqslant u+rnt)\,{\rm d}t;\end{equation}
use that the time till the first event is exponentially distributed with mean $(\lambda n+\vartheta)^{-1}$, and that the first event is a default with probability $\lambda n/(\lambda n+\vartheta)$. 

We proceed by analyzing $\psi_n(\gamma)$ using the relation \eqref{RECU}, with the objective to express it in terms of $\psi_{n-1}(\cdot)$. By a change-of-variable $v:=u+rnt$, we obtain
\begin{align*}\psi_n(\gamma) &=\int_0^\infty e^{-\gamma u}  \int_0^\infty \lambda n\, e^{- (\lambda n+\vartheta) t} \,{\mathbb P}(Z_{n-1}+L\geqslant u+rnt)\,{\rm d}t\,{\rm d}u\\
&=\frac{1}{rn} \int_0^\infty e^{-\gamma u}  \int_u^\infty \lambda n\, e^{- (\lambda n+\vartheta) (v-u)/(rn)}\, {\mathbb P}(Z_{n-1}+L\geqslant v)\,{\rm d}v\,{\rm d}u.
\end{align*}
The next step is to swap the order of the integrals, exploiting the fact that the integral over $u$ allows an elementary solution:
\begin{align*}\frac{1}{rn}& \int_0^\infty  \left(\int_0^v e^{-\gamma u} e^{ (\lambda n+\vartheta) u/(rn)} \,{\rm d}u\right)
 \lambda n\, e^{- (\lambda n+\vartheta) v/(rn)}\, {\mathbb P}(Z_{n-1}+L\geqslant v)\,{\rm d}v\\
 &=\frac{\lambda n}{\gamma rn -\lambda n-\vartheta}\int_0^\infty \big(e^{- (\lambda n+\vartheta) v/(rn)}-e^{-\gamma v}\big) {\mathbb P}(Z_{n-1}+L\geqslant v)\,{\rm d}v.\end{align*}
In the last expression, we see an object that resembles a Laplace transform, but observe that it features a complementary cumulative distribution function rather than a density. Recall however the standard identity
\begin{equation}\label{e1}\int_0^\infty e^{-\gamma u} {\mathbb P}(X\geqslant u){\rm d}u =\frac{1}{\gamma}-\frac{1}{\gamma}\int_0^\infty e^{-\gamma u} {\mathbb P}(X\in{\rm d}u)= \frac{1-{\mathbb E}\,e^{-\gamma X}}{\gamma}.\end{equation}
In addition, using integration by parts, for the non-negative random variable $Z_{n-1}$,
\begin{equation}
\label{e2}{\mathbb E}\,e^{-\gamma Z_{n-1}} = \int_0^\infty e^{-\gamma x} {\mathbb P}(Z_{n-1}\in {\rm d}x) = 1-\gamma  \int_0^\infty {\mathbb P}(Z_{n-1}>x)e^{-\gamma x}{\rm d}x= 1-\gamma\psi_{n-1}(\gamma).\end{equation}
By the identity (\ref{e1}), and using the independence between the random variables $Z_{n-1}$ and $L$, we obtain, for any $\gamma\geqslant 0$, with $d_n:=(\lambda n+\vartheta)/(rn)$,
\begin{align*}\psi_n(\gamma) &=\frac{\lambda n}{\gamma rn -\lambda n-\vartheta}\Big(\frac{rn}{\lambda n+\vartheta}\left(1- {\mathbb E}\,e^{-(\lambda n+\vartheta)/(rn)\,(Z_{n-1}+L)}\right)\,-\frac{1}{\gamma}\left(1- {\mathbb E}\,e^{-\gamma\,(Z_{n-1}+L)}\right)\Big)\\
&=\frac{\lambda n}{\lambda n+\vartheta}\frac{1}{\gamma}+\frac{\lambda n}{\gamma rn -\lambda n-\vartheta}\left(\frac{{\mathbb E}\,e^{-\gamma Z_{n-1}}\ell(\gamma)}{\gamma}-\frac{{\mathbb E}\,e^{-d_n Z_{n-1}}\ell(d_n)}{d_n}\right),
\end{align*}
which, by applying (\ref{e2}) and a few elementary algebraic steps, equals
\[\frac{\lambda n}{\lambda n+\vartheta}\frac{1}{\gamma}+\frac{\lambda n}{\lambda n +\vartheta-\gamma r n}\left(B\left(\frac{\lambda n +\vartheta}{rn},\psi_{n-1}\left(\frac{\lambda n +\vartheta}{rn}\right)\right)-B\left(\gamma,\psi_{n-1}(\gamma)\right)\right),\]
where we define
\[B(x,y):= \ell(x)\left(\frac{1}{x}-y\right).\]

Conclude that we have expressed $\psi_n(\cdot)$ in terms of $\psi_{n-1}(\cdot)$, so that we would obtain a recursion if we would have an explicit expression for $\psi_0(\cdot)$. Recall that $\psi_0(\cdot)$ corresponds to ruin in the scenario without any obligor left. Obviously $p_0(u,t)\equiv 0$ for any $u$ and $t$, entailing that $\psi_0(\gamma)\equiv 0$ for any value of $\gamma$.
It means that we can thus recursively compute $\psi_n(\gamma)$.
The theorem below summarizes the findings so far.

\begin{theorem} \label{TH1} For any $\gamma\geqslant 0$ and $n\in{\mathbb N}$, we have the recursion
\[\psi_{n}(\gamma)= \frac{\lambda n}{\lambda n+\vartheta}\frac{1}{\gamma}+\frac{\lambda n}{\lambda n +\vartheta-\gamma r n}\left(B\left(\frac{\lambda n +\vartheta}{r n},\psi_{n-1}\left(\frac{\lambda n +\vartheta}{rn}\right)\right)-B\left(\gamma,\psi_{n-1}(\gamma)\right)\right),\]
where $\psi_0(\gamma)\equiv 0$.
\end{theorem}

\begin{remark}\label{R1a}{\em 
Interestingly, one could interpret the departure of an obligor as a time change: 
the default arrival rate drops from $\lambda n$ to $\lambda (n-1)$, and simultaneously the aggregate income per time unit drops from $rn$ to $r(n-1)$. As a consequence, in the infinite-horizon setting ($\vartheta=0$, that is) the recursion in Theorem \ref{TH1} greatly simplifies. 
}$\hfill\Diamond$\end{remark}

\begin{remark}\label{R1}{\em 
Upon inspecting the above proof, it is readily checked that it has not been used that the income rate is proportional to the number of obligors present; similarly, it is not crucial that the time till the next default when there are still $i$ obligors is exponential with parameter $\lambda i$. This effectively means that we can work with an income rate $r_i$ (rather than $ri$) and a default rate $\lambda_i$ (rather than $\lambda i$) during times that there are $i$ obligors left. We thus obtain the recursion
\[\psi_{n}(\gamma)= \frac{\lambda_n}{\lambda_n+\vartheta}\frac{1}{\gamma}+\frac{\lambda_n}{\lambda_n +\vartheta-\gamma r_n}\left(B\left(\frac{\lambda_n +\vartheta}{r_n},\psi_{n-1}\left(\frac{\lambda_n +\vartheta}{r_n}\right)\right)-B\left(\gamma,\psi_{n-1}(\gamma)\right)\right),\]
where $\psi_0(\gamma)\equiv 0$. It is also remarked that one can make the loss distribution dependent on the number of obligors in the system, by working with the transform $\beta_i(\cdot)$ when there are still $i$ obligors that have not gone into default yet.
}$\hfill\Diamond$\end{remark}

\begin{remark}\label{rem2}{\em
An interesting special case relates to the situation in which $r_n = r$ and $\lambda_n=\lambda$, i.e., the conventional Cram\'er-Lundberg model. Sending $n\to\infty$, one should recover the (transient version of the) Pollaczek-Khinchine formula. As an illustration, we show this computation for $\vartheta=0$, writing $a$ for $\lambda/r$ and assuming that $-a\ell'(0)<1$. We obtain the relation, with the limit of $\psi_n(\cdot)$ being denoted by $\psi(\cdot)$,
\[\psi(\gamma) = \frac{1}{\gamma}+\frac{a}{a-\gamma}\big(B(a,\psi(a))-B(\gamma,\psi(\gamma)\big).\]
It yields, after some elementary algebra, that 
\[1-\gamma\psi(\gamma) = \frac{\gamma}{\gamma-a+a\ell(\gamma)}\ell(a)(1-a\psi(a)).\]
The constant $\ell(a)(1-a\psi(a))$ can be identified by observing that the left-hand side goes to $1$ as $\gamma\downarrow 0$; hence, an application of H\^opital's rule yields that
\[\ell(a)(1-a\psi(a)) =\lim_{\gamma\downarrow 0}\frac{\gamma-a+a\ell(\gamma)}{\gamma} = 1+a\ell'(0).\]
We conclude
\[\psi(\gamma) =\frac{1}{\gamma}-\frac{1+a\ell'(0)}{\gamma-a+a\ell(\gamma)},\]
which directly corresponds to the Pollaczek-Khinchine formula \cite{AsmussenAlbrecher,DM}. Our new results can be thus be seen as a true generalization of the classical results from ruin theory. $\hfill\Diamond$
}\end{remark}

\begin{remark}{\em
The recursion featuring in Thm.\ \ref{TH1} can be made more explicit when working with its generating function. To demonstrate this, we focus on the case of $\vartheta=0$, $r_n=rn$, and $\lambda_n=\lambda n$. We have, again with $a=\lambda/r$,
\[\psi_n(\gamma) = \frac{1}{\gamma} +\frac{a}{a-\gamma}\left(\ell(a)\left(\frac{1}{a}-\psi_{n-1}(a)\right)-\ell(\gamma)\left(\frac{1}{\gamma}-\psi_{n-1}(\gamma)\right)\right).\]
We thus obtain that, using that $\psi_0(\gamma)=0$, 
\begin{align*}
\Psi(z,\gamma)&:= \sum_{n=1}^\infty z^n\psi_n(\gamma)\\
&=\sum_{n=1}^\infty z^n\frac{1}{\gamma}+z\frac{a}{a-\gamma} \sum_{n=1}^\infty z^{n-1}\left(\ell(a)\left(\frac{1}{a}-\psi_{n-1}(a)\right)-\ell(\gamma)\left(\frac{1}{\gamma}\psi_{n-1}(\gamma)\right)\right)\\
&=\frac{z}{1-z}\frac{1}{\gamma}+z\frac{a}{a-\gamma}\left(\ell(a)\left(\frac{1}{a{(1-z)}}-\Psi(z,a)\right)-\ell(\gamma)\left(\frac{1}{\gamma{(1-z)}}-\Psi(z,\gamma)\right)\right) .\end{align*} 
We conclude that \[\Psi(z,\gamma)=\frac{1}{a-\gamma-za\,\ell(\gamma)}\,\left( \frac{z}{1-z} \frac{a-\gamma}{\gamma} + za\, \ell(a) \left(\frac{1}{a{(1-z)}}-\Psi(z,a)\right)-\frac{z}{1-z} \frac{ a\,\ell(\gamma)}{\gamma}\right).\]
We are thus left with determining $\Psi(z,a)$. 
For $a$ and $z$ fixed there is a unique positive $\gamma\equiv \gamma(z,a)$ for which the denominator equals 0 (as follows from the fact that $\nu(\gamma):=a-\gamma-za\,\ell(\gamma)$ is concave with $\nu(0)=a(1-z)>0$ and $\nu(\gamma)\to-\infty$ as $\gamma\to\infty$). We therefore have that in $\gamma\equiv \gamma(z,a)$ the numerator should equal $0$ as well. 
This leads to
\begin{align*}\Psi(z,a) &= \frac{1}{a{(1-z)}}+\frac{1}{1-z}\frac{1}{\gamma(z,a)\,\ell(a)}\left(\frac{a-\gamma(z,a)}{a} -\ell(\gamma(z,a))\right)\\
&=\frac{1}{a{(1-z)}}+{\frac{a-\gamma(z,a)-a\,\ell(\gamma(z,a))}{(1-z)\,a\gamma(z,a)\,\ell(a)}}.\end{align*}
Combining the above, we have thus identified \[\Psi(z,\gamma)=\frac{z}{1-z}\frac{1}{a-\gamma-za\,\ell(\gamma)}\left(\frac{a-\gamma -a\,\ell(\gamma)}{\gamma}-\frac{a-\gamma(z,a)-a\,\ell(\gamma(z,a))}{\gamma(z,a)}\right).\] 
By multiplying with $(1-z)$, we obtain the transform at a geometrically distributed time with success probability $z$. Sending $z\uparrow 1$, and realizing that $\gamma(1,a)=0$, we recover the stationary result discussed in Remark \ref{rem2}. 
 $\hfill\Diamond$ }                                                                                                                                                                                                                                                                                                                                                                                                                                                                                                                                                                                                                                                                                                                                                                                                                                                                                                                                                                                                                                                                                                                                                                                                                                                                                                                                                                                                                                                                                                                                                                                                                                                                                                                                                                                                                                                                                                     
\end{remark}

\section{Asymptotics, efficient simulation, and uniform bound}\label{asy}
The previous section provides us with a way of computing $p_n(u,t).$ Here one should realize that $\psi_n(\gamma)$ is a (double) transform, so that numerical Laplace inversion needs to be applied in order to evaluate $p_n(u,t)$. Over the past decades significant progress has been made in the domain of Laplace inversion; see for instance the fast, accurate, and generally applicable algorithms described in\ \cite{AW,dI}. If one wishes to avoid numerical inversion, two frequently used alternatives are  (i)~asymptotic techniques, and (ii)~simulated-based estimation. In approach (i), one scales one or more of the model parameters, and aims at finding an explicit expression for the quantity under study (in our case the ruin probability) in the regime that this scaling parameter grows large. Approach (ii) has the intrinsic drawback that, in order to obtain reliable estimates in the domain of small ruin probabilities, many runs are needed. These issues can be remedied by simulating under a suitably chosen alternative measure rather than the actual one, and correcting the simulation output by likelihood ratios; this method is known as importance sampling. 

In this section we present a series of results that help to quantify the ruin probability $p_n(u,t)$ without the need to resort to numerical inversion. Our findings come in three flavors. In the first place we find, for a given $u$ and $t$, the asymptotics of $p_n(nu,t)$ as $n$ grows large; i.e., we scale the initial capital level by the initial number of obligors.
Secondly, we derive a uniform upper bound on $p_n(u,t)$, comparable to the well-known Lundberg inequality for the conventional Cram\'er-Lundberg model. Finally, we develop a provably efficient importance-sampling based simulation algorithm. 
Importantly, in this section we can lift the assumption of exponentially distributed time-to-defaults.

\subsection{Notation and preliminaries}
Throughout this entire section we let the times-to-default $T_1,\ldots,T_n$ be non-negative i.i.d.\ random variables, with density $f(\cdot)$ and distribution function $F(\cdot)$, distributed as a generic random variable $T$.
Let $Z_n(t)$ be the net {cumulative loss amount} at time $t\geqslant 0$, given that at time $0$ there are $n\in{\mathbb N}$ obligors present. We denote, for $i=1,\ldots,n$ and $t\geqslant 0$, by $W_i(t)$ the net {cumulative loss amount} of the $i$-th obligor at time $t$. By distinguishing between the scenario that obligor $i$ has gone into default at time $t$ and its complement, we can write $W_i(t)$ as
\begin{equation}\label{weetje}W_i(t):=1_{\{T_i\leqslant t\}}L_i-r\min\{T_i,t\}.\end{equation}
We define the moment generating function ${\mathbb E}\, e^{\alpha L}$ of the loss $L$ by $\bar\ell(\alpha):=\ell(-\alpha)$. Then, due to fact that the obligors are statistically identical,
\[{\mathbb E}\,e^{\alpha Z_n(t)} =\left({\mathbb E}\,e^{\alpha W_1(t)} \right)^n.\]
In addition, we can compute the moment generating function of the net {loss} amount of a single obligor at time $t$. By conditioning on the time-to-default, using \eqref{weetje},
\begin{align*}
\omega_t(\alpha):={\mathbb E}\,e^{\alpha W_1(t)}&=\bar\ell(\alpha)\int_0^t f(s) e^{-r\alpha s}{\rm d}s  + 
e^{-r\alpha t}\int_t^\infty f(s) {\rm d}s\\
&=\bar\ell(\alpha)\int_0^t f(s) e^{-r\alpha s}{\rm d}s  + 
e^{-r\alpha t} (1-F( t)).
\end{align*}
For instance, in the special case that the time-to-defaults are exponentially distributed with mean $\lambda^{-1}$, we have
\[\omega_t(\alpha)=\left(1-e^{-(\lambda+r\alpha)t}\right) \frac{\lambda}{\lambda+r\alpha}\bar\ell(\alpha) + e^{-(\lambda+r\alpha)t}.\]

\subsection{Large-deviations asymptotics}\label{subsec_LD}
The goal of this subsection is to establish a limit theorem for our ruin probability, given that we start with $n$ obligors and an initial capital reserve level $nu>0$, as $n$ grows large. In other words, we analyze how the probability
\begin{equation}\label{qn}q_n(t):=  p_n(nu,t)={\mathbb P}\left(\exists s\in[0,t]: Z_n(s) \geqslant nu\right)= {\mathbb P}\left(\exists s\in[0,t]:\sum_{i=1}^n W_i(s) \geqslant nu\right)\end{equation}
behaves as $n\to\infty.$
We do so under the evident `rarity condition' that, for all $t\geqslant 0$, ${\mathbb E}Z_n(t)$ is smaller than $nu$, or, equivalently, 
\[\sup_{t\geqslant 0}\big({\mathbb P}(T\leqslant t) \,{\mathbb E}L - r\,{\mathbb E}\min \{T,t\} \big) < u,\]
where we use that ${\mathbb E}W_1(t)=\omega'_t(0) = {\mathbb P}(T\leqslant t) \,{\mathbb E}L - r\,{\mathbb E}\min \{T,t\}.$
We start by establishing a lower bound. The underlying principle is that the probability of a union of events is bounded from below by the probability of the most likely event among them. This entails that, for any $s\in[0,t]$ we have that $q_n(t)\geqslant \check q_n(s)$, where
\[\check q_n(s):= {\mathbb P}\left(\sum_{i=1}^n W_i(s) \geqslant nu\right).\]
Define the Legendre transform pertaining to $W_1(s)$:
\[I(s):=\sup_{\alpha}\left(\alpha u - \log \omega_s(\alpha)\right).\]
Because of the rarity condition $\omega'_s(0) <u$ for all $s\geqslant 0$, we can restrict ourselves to maximizing over $\alpha>0$ only;
we define $\alpha^\star(s):=  \arg\sup_{\alpha}\left(\alpha u - \log \omega_s(\alpha)\right).$
By Cram\'er's theorem \cite{DZ}, we immediately have that, for any $s\in[0,t]$,
\begin{equation}\label{LBS}\liminf_{n\to\infty}\frac{1}{n}\log q_n(t) \geqslant\liminf_{n\to\infty}\frac{1}{n}\log \check q_n(s) =
-\alpha^\star(s) u + \log \omega_s(\alpha^\star(s))=-I(s).\end{equation}
We also define 
\[t^\star:= \arg\inf_{s\in[0,t]}I(s),\]
which has the informal interpretation of the most likely time $Z_n(\cdot)$ exceeds $nu.$
From the fact that the lower bound \eqref{LBS} applies for any $s\in[0,t]$, we thus obtain that
\[\liminf_{n\to\infty}\frac{1}{n}\log q_n(t) \geqslant-\inf_{s\in[0,t]}I(s) = -I(t^\star).\]

We proceed by proving that $-I(t^\star)$ is also an upper bound on the decay rate of $q_n(t)$. The first step is to realize that ruin occurs at the default time of one of the $n$ obligors. As a consequence, we can rewrite $q_n$ in terms of the union of $n$ events:
\[q_n(t)={\mathbb P}\left(\exists j\in\{1,\ldots,n\}: T_j\in[0,t],  \sum_{i=1}^n W_i(T_j) \geqslant nu\right),\]
instead of the union of uncountable many events featuring in the representation (\ref{qn}).
By the union bound, we obtain that this probability $q_n(t)$ is majorized by $n\hat q_n(t)$, where
\[\hat q_n(t):={\mathbb P}\left(T_1\in[0,t],  \sum_{i=1}^n W_i(T_1) \geqslant nu\right).\]
As $n^{-1} \log n \to 0$, it suffices to prove that $\limsup_{n\to\infty} n^{-1} \log \hat q_n(t)\leqslant -I(t^\star).$ To this end,  by conditioning on $T_1$,
\[\hat q_n(t) = \int_0^t f(s) \,{\mathbb P}\left( \sum_{i=2}^{n} W_i(s) +L_1 -rs \geqslant nu\right) {\rm d}s.\]
Then observe that the $W_i(T_1)$ are dependent, but once conditioned on $T_1=s$ they have become independent.  The next step is to apply the Markov inequality: for any $\alpha\geqslant 0$, with $L_1$ being independent from $W_2(s),\ldots,W_{n}(s)$,
\begin{align*}
{\mathbb P}\left(
 \sum_{i=2}^{n} W_i(s) +L_1 -rs \geqslant nu
 \right)&=
{\mathbb P}\left( \exp\left({\alpha \sum_{i=2}^{n} W_i(s) +\alpha L_1}\right) \geqslant \exp({\alpha (nu+rs)})\right)  \\
&\leqslant (w_s(\alpha))^{n-1}\bar\ell(\alpha) \,e^{-\alpha(nu+rs)}
\leqslant (w_s(\alpha))^{n-1}\bar\ell(\alpha) \,e^{-\alpha(n-1)u}
.\end{align*}
Upon combining the above, we have thus found that for any $\alpha(\cdot)\geqslant 0$,
\[\limsup_{n\to\infty}\frac{1}{n}\log \hat q_n(t) \leqslant \limsup_{n\to\infty}\frac{1}{n}\log 
\int_0^t f(s) \,(w_s(\alpha(s)))^{n-1} \bar\ell(\alpha(s))\,e^{-\alpha(s)\,(n-1)u} {\rm d}s.\]
Recall that, for any $t\geqslant 0$,  $I(t)=\alpha^\star(t) u - \log \omega_t(\alpha^\star(t)).$
Plugging in $\alpha(\cdot)=\alpha^\star(\cdot),$ we thus obtain, in the second inequality using the definition of $t^\star$,
\begin{align}\nonumber\limsup_{n\to\infty}\frac{1}{n}\log \hat q_n (t)&\leqslant \limsup_{n\to\infty}\frac{1}{n}\log 
\int_0^t f(s)\, \bar\ell(\alpha^\star(s))\,e^{-(n-1)I(s)} {\rm d}s\\ \nonumber
&\leqslant \limsup_{n\to\infty}\frac{1}{n}\log 
\int_0^t f(s)\, \bar\ell(\alpha^\star(s))\,e^{-(n-1)I(t^\star)} {\rm d}s\\
&=-I(t^\star)+ \limsup_{n\to\infty}\frac{1}{n}\log 
\int_0^t f(s)\, \bar\ell(\alpha^\star(s))\, {\rm d}s.\label{laatste}
\end{align}
Observe that we are done if we succeed in proving that the second term in \eqref{laatste} is $0$, for which it suffices to prove that the integral appearing in this term is finite.
To this end, first observe that, with $\tau(\alpha):= {\mathbb E}\,e^{\alpha T}$, 
\[\lim_{t\to\infty} \omega_t(\alpha) = \bar\ell(\alpha) \tau(-r\alpha)=:\Delta(\alpha),\]
so that $\alpha^\star(\infty)$ solves $\Delta'(\alpha)/\Delta(\alpha) = u$. 

\begin{assumption} \label{ASS1} The function $\alpha^\star(\cdot)$ is bounded on $[0,t]$.
\end{assumption}

Under Assumption \ref{ASS1}, we have $\sup_{s\in[ 0,t]}\alpha^\star(s)\leqslant M$ for some finite $M$. Note that this holds whenever the function $\alpha^\star(\cdot)$ is continuous, whereas in case $t=\infty$ we additionally require $\alpha^\star(\infty)<\infty$. 
With this assumption in place and using that $\alpha\mapsto \bar\ell(\alpha)$ is increasing, we conclude that 
\[\int_0^t f(s)\, \bar\ell(\alpha^\star(s))\, {\rm d}s \leqslant \bar\ell(M) \int_0^t f(s)\, {\rm d}s \leqslant \bar\ell(M)<\infty.\]
Summarizing, we have shown
\[\limsup_{n\to\infty}\frac{1}{n}\log q_n(t) \leqslant -I(t^\star).\]
We have arrived at the following result.
\begin{theorem} \label{TH31} As $n\to\infty$, under Assumption $\ref{ASS1}$, 
\[\frac{1}{n}\log q_n(t) \to -I(t^\star).\]
\end{theorem}

\subsection{Efficient simulation}\label{subsec_is}
As the above theorem only provides us with the logarithmic asymptotics of $q_n$, it is inherently imprecise. For instance, if the true asymptotic shape of $q_n$ is $n^\alpha \,\exp({-nI(t^\star)})$ for some $\alpha\in{\mathbb R}$, or $\exp(n^\eta) \exp({-nI(t^\star)})$
for some $\eta\in (0,1)$, the effect of the $\alpha$ and $\eta$ is not visible. 
One can get accurate estimates in an efficient way, however, applying importance sampling.  Below we present an importance sampling algorithm, which we prove to be logarithmically efficient. 

The key idea is that we decompose our rare-event probability $q_n$ into $n$ rare-event probabilities, which we will be dealing with separately. We write
\begin{equation}\label{sum}q_n(t) = \sum_{j=1}^n q_{nj}(t),\end{equation}
where 
\[q_{nj}(t):={\mathbb P}\left({\mathscr F}_j\right),\:\:\:\:
{\mathscr F}_j:=
{\mathscr E}_j\cap \bigcap_{i=1}^{j-1} {\mathscr E}_i^{\rm c},\:\:\:\:
{\mathscr E}_j:= \left\{T_j\in[0,t],
 \sum_{i\not=j} W_i(T_j) +L_j -rT_j \geqslant nu
\right\};\]
the validity of \eqref{sum} is due to the events ${\mathscr F}_j$ being (by construction) disjoint, while the union of the ${\mathscr E}_j$ equals the union of the ${\mathscr F}_j$. The problem of efficiently estimating $q_n(t)$ thus reduces to the problem of efficiently estimating each of the $q_{nj}(t)$ (and adding up all the resulting estimates). 

Fix a $j\in\{1,\ldots,n\}$ and focus on the estimation of $q_{nj}$. We now define an importance sampling probability measure ${\mathbb Q}$.
\begin{itemize}
\item[$\circ$]
Under ${\mathbb Q}$ the density of $T_j$ remains $f(\cdot)$. 
\item[$\circ$] Conditionally on $T_j=s$, the moment generating function of $L_j$ becomes
\[\bar\ell^{\mathbb Q}(\alpha) = \frac{\bar\ell(\alpha+\alpha^\star(s)}{\bar\ell(\alpha^\star(s))}.
\]
Sampling $L_j$ from ${\mathbb Q}$ amounts to sampling from an exponentially twisted version of the actual distribution. This is a standard procedure in applied probability; for many frequently used distributions the twisted distribution remains within the same class of distributions, but with different parameters. For instance, the $\alpha$-twisted version of an exponentially distributed random variable with parameter $\mu$ corresponds to an exponentially distributed random variable with parameter $\mu-\alpha$ (requiring that $\alpha\in[0,\mu)$).
\item[$\circ$] Conditionally on $T_j=s$, the moment generating function of $W_i(s)$ (for $i\not= j$) becomes
\begin{equation}\label{alp}\omega_s^{\mathbb Q}(\alpha):= \frac{\omega_s(\alpha+\alpha^\star(s))}{\omega_s(\alpha^\star(s))}.\end{equation}
To decide whether the event ${\mathscr F}_j$ applies, we have to sample the default times $T_i$ and (if $T_i<t$) the losses $L_i$, for $i\not=j$, in accordance with \eqref{alp}. This can be done as follows. By (\ref{alp}), the exponentially twisted version of $W_i(s)$ has the moment generating function
\begin{align*}\omega_s^{\mathbb Q}(\alpha)=
\frac{1}{\omega_s(\alpha^\star(s))}&\left(\int_0^s f(v)\, e^{-(\alpha+\alpha^\star(s))\, rv} \bar\ell(\alpha+\alpha^\star(s)){\rm d}v \:+\right.\\ &\hspace{3mm}\left.\int_s^\infty f(v)\, e^{-(\alpha+\alpha^\star(s))\, rs} {\rm d}v\right).\end{align*}
From this identity we observe that the $L_i$ can be sampled from a distribution with moment generating function $\bar\ell^{\mathbb Q}(\cdot)$, as defined above, whereas the density $f^{\mathbb Q}(\cdot)$ of the $T_i$ (for $i\not=j$) becomes
\[f^{\mathbb Q}(v) =\frac{1}{\omega_s(\alpha^\star(s))}f(v)\left( e^{-\alpha^\star(s)\, rv} \bar\ell(\alpha^\star(s))1_{\{v\leqslant s\}} +  e^{-\alpha^\star(s)\, rs} 1_{\{v>s\}}\right).\]
\end{itemize}
We proceed by detailing the importance-sampling based simulation procedure, and establishing its asymptotic efficiency. 
To this end, we first observe that a generic sample of the likelihood ratio, say ${\mathscr L}_j$, has the form
\[
{e^{-\alpha^\star(T_j)\,L_j}}\cdot{\bar\ell(\alpha^\star(T_j))}
\prod_{i\not=j} \left({e^{-\alpha^\star(T_j)\,W_i(T_j)}}\cdot{\omega_{T_j}(\alpha^\star(T_j))}\right).\]
Recall that on the event ${\mathscr F}_j$ we have $ \sum_{i\not= j} W_i(T_j) +L_j -rT_j \geqslant nu$. As a consequence, on the event ${\mathscr F}_j$ the likelihood ratio ${\mathscr L}_j$ is majorized by
\begin{align*}e^{-\alpha^\star(T_j)(nu + r T_j)}\cdot{\bar\ell(\alpha^\star(T_j))}&\cdot\big(\omega_{T_j}(\alpha^\star(T_j))\big)^{n-1}\\
&\leqslant e^{-\alpha^\star(T_j)(n-1)u}\cdot{\bar\ell(\alpha^\star(T_j))}\cdot\big(\omega_{T_j}(\alpha^\star(T_j))\big)^{n-1}\\
&=\bar\ell(\alpha^\star(T_j))\,e^{-(n-1)\,I({T_j})} \leqslant \bar\ell(M)\,e^{-(n-1)\,I({T_j})}\\
&\leqslant  \bar\ell(M)\,e^{-(n-1)\,I({t^\star})},
\end{align*}
with $M$ as defined in Section \ref{subsec_LD} (where we let Assumption \ref{ASS1} be in force).
We thus find that, with ${\mathscr I}_j$ denoting the indicator function of ${\mathscr F}_j$, the almost-sure inequality
${\mathscr L}_j\,{\mathscr I}_j \leqslant \bar\ell(M)\,e^{-(n-1)\,I({t^\star})}$, and therefore
\[\sum_{j=1}^n {\mathscr L}_j\,{\mathscr I}_j \leqslant n \,\bar\ell(M)\,e^{-(n-1)\,I({t^\star})}.\]
Evidently, to obtain an estimator with good precision, we have to repeat the above experiment sufficiently often. Suppose, for each $j\in\{1,\ldots,n\}$, we perform $N\in{\mathbb N}$ independent trials. The corresponding likelihood ratios are denoted by 
${\mathscr L}_{j,k}$, and the indicator functions are ${\mathscr I}_{j,k}$, with $j\in\{1,\ldots,n\}$ and $k\in\{1,\ldots,N\}$. 
Our estimator thus becomes  
\[\xi_N:=\frac{1}{N} \sum_{k=1}^N \sum_{j=1}^n {\mathscr L}_{j,k}\,{\mathscr I}_{j,k},\]
which is (by construction) unbiased.
The next step is to analyze the performance of this estimator. To this end, we observe in relation to its second moment that
\[{\mathbb E}_{\mathbb Q}\left(\left(\sum_{j=1}^n {\mathscr L}_{j}\,{\mathscr I}_{j}\right)^2\right) \leqslant n^2 \,(\bar\ell(M))^2\,e^{-2(n-1)\,I({t^\star})},\]
with ${\mathbb E}_{\mathbb Q}(\cdot)$ denoting expectation under ${\mathbb Q}$.
We find the following upper bound for the second moment: 
\[\limsup_{n\to\infty}\frac{1}{n}\log {\mathbb E}_{\mathbb Q}\left(\left(\sum_{j=1}^n {\mathscr L}_{j}\,{\mathscr I}_{j}\right)^2\right) \leqslant -2 I(t^\star).\]
By Theorem \ref{TH31}, and in addition using that variances are non-negative, we also have the corresponding lower bound:
\begin{align*}
\liminf_{n\to\infty}\frac{1}{n}\log {\mathbb E}_{\mathbb Q}\left(\left(\sum_{j=1}^n {\mathscr L}_{j}\,{\mathscr I}_{j}\right)^2\right)& \geqslant \liminf_{n\to\infty}\frac{2}{n}\log {\mathbb E}_{\mathbb Q}\left(\sum_{j=1}^n {\mathscr L}_{j}\,{\mathscr I}_{j}\right) \\&= \liminf_{n\to\infty}\frac{2}{n}\log q_n(t) = -2 I(t^\star).\end{align*}
The above bounds lead to the following conclusion, which in practical terms entails that the number of runs needed to obtain an estimate of a given relative precision, grows sub-exponentially in $n$. For the definition of logarithmic efficiency, and related performance notions in rare-event simulation, we refer to \cite[Ch. VI]{AG}.
\begin{theorem} Under Assumption $\ref{ASS1}$, the estimator $\xi_N$ is logarithmically efficient as $N\to\infty$.
\end{theorem}

\subsection{Uniform bound} \label{Unif}
Intrinsic drawbacks of the large-deviations asymptotics is that they only kick in for large $n$, and they provide us with the decay rate only. This motivates the search for a uniform upper bound on the ruin probability $p_n(u,t)$. The result is a Lundberg-type inequality derived along the same lines was done in \cite[Section XIII.5a]{AsmussenAlbrecher} for the conventional Cram\'er-Lundberg model in which claims (or losses in the credit context) arrive according to a fixed-intensity Poisson process. We focus on the situation that when there are $n$ obligors the time to the first default is exponentially distributed with mean $\lambda_n^{-1}$ and the income rate is $r_n.$ Let $\gamma_n$ be the positive solution for $\gamma$ in 
\[\bar\ell(\gamma) \frac{\lambda_n}{\lambda_n+\gamma r_n}=1.\]

\vb

\begin{theorem}\label{th_upper}
Suppose that $\gamma_n$ is non-increasing in $n$. Then 
\[p_n(u,t)\leqslant p_n(u,\infty)\leqslant e^{-\gamma_n u}.\]
\end{theorem}
\begin{proof} It is evident that $p_n(u,t)\leqslant p_n(u,\infty)$.
Let $Y_n$ be distributed as $L-r_n\,T_1$, where $T_1$ is assumed exponentially distributed with mean $\lambda_n^{-1}$ (independent of $L$).
Conditioning on $Y_n$ immediately yields
\[p_n(u,\infty)=\mathbb{P}(Y_{n}>u)+\int_{-\infty}^u p_{n-1}(u-y,\infty){\mathbb P}({Y_{n}}\in{\rm d}y).\]
We claim that this implies $p_n(u,\infty)\leqslant e^{-\gamma_n u}$. The proof is by induction. First note that the claim holds true for $n=0$ as $p_0(u,\infty)=0$ for all $u> 0$. Assuming the inequality holds true for $n-1$,\begin{align*}
p_{n}(u,\infty)&\leqslant\mathbb{P}(Y_{n}>u)+\int_{-\infty}^u e^{-\gamma_{n-1} (u-y)}\,{\mathbb P}({Y_{n}}\in{\rm d}y)\\
&\leqslant\mathbb{P}(Y_{n}>u)+\int_{-\infty}^u e^{-\gamma_{n} (u-y)}\,{\mathbb P}({Y_{n}}\in{\rm d}y)\\
&\leqslant e^{-\gamma_{n} u}\int^{\infty}_u e^{\gamma_{n} y}\,{\mathbb P}({Y_{n}}\in{\rm d}y)+\int_{-\infty}^u e^{-\gamma_{n} (u-x)}\,{\mathbb P}({Y_{n}}\in{\rm d}y)\\
&=e^{-\gamma_{n} u}\,{\mathbb E}\,e^{\gamma_{n} Y_n}=e^{-\gamma_{n} u}\,\bar\ell(\gamma_n) \frac{\lambda_n}{\lambda_n+\gamma_n r_n}=e^{-\gamma_n u},
\end{align*}
where in the second inequality it has been used that that $\gamma_n$ is non-increasing in $n$.
\end{proof}

\begin{remark}{\em 
In the special case the default arrival intensity $\lambda_n$ and the income rates $r_n$ are linear in the number of obligors $n$, it is readily checked that $\gamma_n$ does not depend on $n$. As a consequence, also the upper bound derived above does not depend on $n$. $\hfill\Diamond$
}\end{remark}

\section{Non-default losses, Markov modulation, \\Brownian perturbations, and multiple groups} \label{ext}

In this section we consider four important extensions of our base model.
\begin{itemize}
\item[$\circ$] In the first extension there are both losses due to defaults (reducing the number of obligors by one) and losses that do not correspond to defaults (leaving the number of obligors unchanged).
\item[$\circ$] Then we consider a model in which the dynamics are affected by a Markovian background process, thus creating dependence between the individual obligors.
\item[$\circ$] We proceed by analyzing a model in which the cumulative process between jumps behaves as a Brownian motion (rather than being linear).
\item[$\circ$] Finally we discuss an extension that allows heterogeneous obligors (by working with multiple groups). 
\end{itemize}
Note that, as opposed to the analysis presented in the previous section, in this section we let the default times be exponentially distributed. In principle, the four generalizations introduced above can be combined, but to keep the presentation as transparent as possible we have decided to discuss them separately.

\subsection{Non-default losses} In this subsection we consider the following extension of the model analyzed in Section \ref{S3} (or, actually, the more general one featured in Remark \ref{R1}). Next to losses due to defaults (happening at a Poisson rate $\lambda_n$ with the losses having Laplace transform $\ell(\cdot)$ when $n$ obligors are present) there are losses that do {\it not} correspond to defaults (happening at a Poisson rate $\lambda^\circ_n$ with the losses having Laplace transform $\ell^\circ(\cdot)$ when $n$ obligors are present).

We again start our derivations by conditioning on the first event, being the first default, the first loss (not leading to default), or the expiration of the exponential clock. If a default happens first, then we can still reach ruin, but now with $n-1$ obligors and an adapted initial reserve. In case the first event is a loss which does not correspond to a default, then we can still reach ruin with $n$ obligors but an adapted initial reserve. If the exponential clock expires, then we will not be facing ruin.

This idea can be formalized as follows. With $L^\circ$ denoting a generic random variable corresponding with a non-default loss, we obtain the relation
\begin{align*}p_n(u) &= \int_0^\infty e^{- (\bar \lambda_n+\vartheta) t}\Big(
\lambda_n\, {\mathbb P}(Z_{n-1}+L\geqslant u+r_nt)+ \lambda^\circ_n\,  {\mathbb P}(Z_n+L^\circ\geqslant u+r_nt)\Big){\rm d}t.
\end{align*}
Going through the same type of computations as those relied on in Section \ref{S3}, we end up with a relation between $\psi_n(\cdot)$ and $\psi_{n-1}(\cdot)$. More specifically, 
for any $\gamma\geqslant 0$, using the notation $\bar{\lambda}_n=\lambda_n+\lambda_n^\circ$, we find that
\begin{align}\nonumber\psi&_{n}(\gamma)= \frac{\bar\lambda_n}{\bar\lambda_n+\vartheta}\frac{1}{\gamma}\:+\\\nonumber&\:\:\:\frac{\lambda_n}{\bar\lambda_n+\vartheta-\gamma r_n}\left(B\left(\frac{\bar\lambda_n+\vartheta}{r_n},\psi_{n-1}\left(\frac{\bar\lambda_n+\vartheta}{r_n}\right)\right)-B\left(\gamma,\psi_{n-1}(\gamma)\right)\right)\:+\\&\:\:\:\frac{\lambda_n^\circ}{\bar\lambda_n+\vartheta-\gamma r_n}\left(B^\circ\left(\frac{\bar\lambda_n+\vartheta}{r_n},\psi_{n}\left(\frac{\bar\lambda_n+\vartheta}{r_n}\right)\right)-B^\circ\left(\gamma,\psi_{n}(\gamma)\right)\right),\label{psin}\end{align}
where $B^\circ(\cdot,\cdot)$ is defined as $B(\cdot,\cdot)$ but with $\ell(\cdot)$ replaced by $\ell^\circ(\cdot)$. 
Unfortunately, this relation between $\psi_n(\cdot)$ and $\psi_{n-1}(\cdot)$ cannot be directly written in terms of an explicit recursion (as opposed to the model without non-default losses; see Theorem \ref{TH1}). The $\psi_n(\cdot)$, however, can still be found recursively, using the following procedure.

To this end, we start by defining the (yet unknown) constants
\[A_n:=B^\circ\left(\frac{\bar\lambda_n+\vartheta}{r_n},\psi_{n}\left(\frac{\bar\lambda_n+\vartheta}{r_n}\right)\right).\]
Then, using that $\psi_0(\cdot)\equiv 0$, observe that $\psi_1(\gamma)$ obeys 
\begin{align}\nonumber\psi_1(\gamma) = \:&\frac{\bar\lambda_1}{\bar\lambda_1+\vartheta}\frac{1}{\gamma}+\frac{\lambda_1}{\bar\lambda_1+\vartheta-\gamma r_1}\left(\frac{r_1}{\bar\lambda_1+\vartheta}\ell\left(\frac{\bar\lambda_1+\vartheta}{r_1}\right)-\frac{\ell(\gamma)}{\gamma}\right)\:+\\
&\:\:\:\frac{\lambda_1^\circ}{\bar\lambda_1+\vartheta-\gamma r_1}\left(B^\circ\left(\frac{\bar\lambda_1+\vartheta}{r_1},\psi_{1}\left(\frac{\bar\lambda_1+\vartheta}{r_1}\right)\right)-B^\circ\left(\gamma,\psi_{1}(\gamma)\right)\right).\label{psi1_nd}\end{align}
We can rewrite (\ref{psi1_nd}), for a known function $F(\cdot)$, as
\[\psi_1(\gamma) = F(\gamma) +\frac{\lambda_1^\circ}{\bar\lambda_1+\vartheta-\gamma r_1}\left(A_1-\ell^\circ(\gamma)\left(\frac{1}{\gamma}-\psi_1(\gamma)\right)\right),\]
which can be rearranged to
\[1-\gamma \psi_1(\gamma) = 1 -\frac{\gamma F(\gamma)(\bar\lambda_1+\vartheta-\gamma r_1)+\gamma\lambda_1^\circ A_1-\lambda_1^\circ\ell^\circ(\gamma)}{\bar\lambda_1-\lambda_1^\circ\ell^\circ(\gamma)+\vartheta-\gamma r_1}.\]
As we know that $1-\gamma \psi_1(\gamma)$ is a Laplace transform, its value should be between 0 and 1 for any $\gamma\geqslant 0$.
Hence, any zero of the denominator is necessarily also a zero of the numerator. It is standard to verify that the numerator has a single positive zero, say $\bar\gamma$. Then it follows that 
\[A_1 = \frac{\ell^\circ(\bar\gamma)}{\bar{\gamma}}-F(\bar\gamma)\frac{\bar\lambda_1+\vartheta-\bar\gamma r_1}{\lambda_1^\circ}.\]
Now that we have found $A_1$ and hence $\psi_1(\gamma)$, we can identify $A_2$ and $\psi_2(\gamma)$ along the same lines:
we first express $\psi_2(\gamma)$ in terms of $A_2$ using \eqref{psin}, and then identify $A_2$ using that the zero of the denominator (which we know to equal $\bar\lambda_2 -\lambda_2^\circ\ell^\circ(\gamma)+\vartheta-\gamma r_2$) is a zero of the numerator as well. Continuing this procedure, all $\psi_n(\gamma)$ (and constants $A_n$) can be found.

\subsection{Markov modulation}
In the models discussed so far the individual obligors are independent. In reality they may be affected by common external factors, to be thought of as the `state of the economy', and hence behave dependently. In this subsection we consider a model in which a particular dependence structure is incorporated, through the mechanism of Markov modulation (also known as regime-switching). 

We start by describing the model. Let $(J(t))_{t\geqslant 0}$ be an irreducible continuous-time Markov process living on $\{1,\ldots,d\}$. We denote by $q_{jk}\geqslant 0$ (for $j\not=k$) the transition rate from state $j$ to state $k$, and $q_j:=-q_{jj}=\sum_{k\not=j} q_{jk}$. Let $r_{nj}$ be the rate at which the surplus process increases when there are $n$ obligors and the background process is in state $j$, let $\lambda_{nj}$ be the corresponding hazard rate of the time to the next default, and let $\ell_j(\cdot)$ be the Laplace transform of the loss (with the associated generic random variable being denoted by $L_j$). 

Let $T_n$ be the minimum of the time of the first default and the expiration of an exponential clock of rate $\vartheta$. Denote by
\[R(T_n):=\int_0^{T_n} r_{nJ(t)}{\rm d}t\]
the increase of the surplus process till $T_n$. We start by analyzing the distribution of $R(T_n)$ through the object
\[F_{i,j,n}(x):= {\mathbb P}_i(R(T_n)\geqslant x, J(T_n)=j) := {\mathbb P}(R(T_n)\geqslant x, J(T_n)=j\,|\, J(0)=i).\]
Using the standard `Markovian reasoning', i.e., by distinguishing between all possible events in a (small) time interval of length $\Delta$ and using the memory-less property, we obtain the relation, as $\Delta\downarrow 0$,
\[F_{i,j,n}(x) = \sum_{k\not = j}F_{i,k,n}(x)\,q_{kj}\Delta + F_{i,j,n}(x-r_j\Delta)\big(1- (q_j+\lambda_{nj}+\vartheta)\big)+o(\Delta).\]
Subsequently subtracting $F_{i,j,n}(x-r_j\Delta)$ from both sides, dividing by $\Delta$ and taking the limit $\Delta\downarrow 0$, we end up with a system of linear differential equations:
\[F'_{i,j,n}(x) = \sum_{k=1}^d F_{i,k,n}(x)\,q_{kj} + F_{i,j,n}(x)\,(\lambda_{nj}+\vartheta).\]
For given $i$ and $n$, this is a system of $d$ coupled linear differential equations, that can be solved in the standard manner; the resulting structure depends on the multiplicities of the eigenvalues. In the sequel we assume that its solution is such that the corresponding density obeys
\[{\mathbb P}_i(R(T_n)\in {\rm d} x, J(T_n)=j) = \sum_{k=1}^d \xi_{i,j,k,n} e^{-\zeta_{k,n}x},\]
but a similar analysis can be done if the terms in the right-hand side of the previous display also involve polynomial factors (as a consequence of the multiplicities of some of the eigenvalues being larger than one). 

The key observation is the identity
\begin{align*}{\mathbb P}_i(Z_n\geqslant u) &= \sum_{j=1}^d \frac{\lambda_{nj}}{\lambda_{nj}+\vartheta}\int_0^\infty {\mathbb P}_j(Z_{n-1}\in {\rm d}z) {\mathbb P}_i(L_j\geqslant R(T_n)+u-z, J(T_n)=j)\\
&= \sum_{j=1}^d \frac{\lambda_{nj}}{\lambda_{nj}+\vartheta}\int_0^\infty\int_0^\infty {\mathbb P}_j(Z_{n-1}\in {\rm d}z) {\mathbb P}(L_j\geqslant x+u-z) \sum_{k=1}^d \xi_{i,j,k,n} e^{-\zeta_{k,n}x}\,{\rm d}x\\
&= \sum_{j=1}^d \frac{\lambda_{nj}}{\lambda_{nj}+\vartheta}\int_0^\infty {\mathbb P}_j(Z_{n-1}+L_j\geqslant x+u)\sum_{k=1}^d \xi_{i,j,k,n} e^{-\zeta_{k,n}x}\,{\rm d}x
\end{align*}
Therefore, using the by now familiar steps concerning a change-of-variables and swapping the order of integration,
\begin{align*}
\psi_{ni}(\gamma)&:=\int_0^\infty e^{-\gamma u} {\mathbb P}_i(Z_n\geqslant u)\,{\rm d}u\\
&=\sum_{j=1}^d \frac{\lambda_{nj}}{\lambda_{nj}+\vartheta}\int_0^\infty\int_0^\infty e^{-\gamma u} {\mathbb P}_j(Z_{n-1}+L_j\geqslant x+u)\sum_{k=1}^d \xi_{i,j,k,n} e^{-\zeta_{k,n}x}\,{\rm d}x\,{\rm d}u\\
&=\sum_{j=1}^d \frac{\lambda_{nj}}{\lambda_{nj}+\vartheta}\int_0^\infty\int_u^\infty e^{-\gamma u} {\mathbb P}_j(Z_{n-1}+L_j\geqslant v)\sum_{k=1}^d \xi_{i,j,k,n} e^{-\zeta_{k,n}(v-u)}\,{\rm d}v\,{\rm d}u\\
&=\sum_{j=1}^d \frac{\lambda_{nj}}{\lambda_{nj}+\vartheta}\int_0^\infty\sum_{k=1}^d \xi_{i,j,k,n}
\left(\int_0^v e^{-\gamma u} e^{\zeta_{k,n}u}
\,{\rm d}u\right){\mathbb P}_j(Z_{n-1}+L_j\geqslant v) e^{-\zeta_{k,n}v}\,{\rm d}v\\
&=\sum_{j=1}^d \frac{\lambda_{nj}}{\lambda_{nj}+\vartheta}\int_0^\infty\sum_{k=1}^d \xi_{i,j,k,n}
\frac{e^{-\zeta_{k,n}v}-e^{-\gamma v}}{\gamma- \zeta_{k,n}}
{\mathbb P}_j(Z_{n-1}+L_j\geqslant v)\,{\rm d}v.
\end{align*}
From now on we can follow the approach presented in Section \ref{S3}: the last expression in the previous display can be expressed in terms of $\psi_{n-1,j}(\cdot)$, for $j=1,\ldots,d.$ We thus end up with a vector-valued recursion. As the derivation is fully analogous to the one corresponding to the non-modulated case, we omit the details. 

\subsection{Brownian perturbations}

We proceed by making the model more realistic by allowing the process to evolve, between defaults, as Brownian motion rather than a deterministic drift. The parameters of this Brownian motion depend on the number of obligors that have not gone in default yet, say with drift coefficient $r_i$ and variance coefficient $\sigma_i^2$ when there are $i$ obligors left. In this section the time between the $i$-th and $(i+1)$-st default is exponentially distributed with mean $\lambda_i^{-1}$.

Considering a Brownian motion with parameters $r$ and $\sigma^2$ over an interval with exponentially distributed length with mean $\lambda^{-1}$, it is known from Wiener-Hopf theory, that 
\begin{itemize}
\item[$\circ$]
the maximum value $M^+$ achieved is exponentially distributed with the
parameter 
\[\nu^+\equiv \nu^+(r,\sigma^2,\lambda):= \frac{\sqrt{r^2+2\lambda\sigma^2}}{\sigma^2}-\frac{r}{\sigma^2}.\]
\item[$\circ$] the (absolute value of the) amount by which the process goes down after the maximum is achieved until the end of the exponentially distributed interval, say $M^-$, is exponentially distributed with the parameter
\[\nu^-\equiv \nu^-(r,\sigma^2,\lambda):= \frac{\sqrt{r^2+2\lambda\sigma^2}}{\sigma^2}+\frac{r}{\sigma^2}.\]
\item[$\circ$] the random variables $M^+$ and $M^-$ are independent. The rates $\nu^+$ and $\nu^-$ are the roots of the equation
$\lambda+r\alpha-\frac{1}{2}\alpha^2\sigma^2=0$. 
\end{itemize}
Now define $\nu_n^\pm:= \nu^\pm(-r_n,\sigma_n^2,\lambda_n+\vartheta)$; note that the first parameter is $-r_n$ rather than $r_n$, as we consider the event of the cumulative claim process exceeding the value $u$ (i.e., the reserve level dropping below $0$). 
As before, we set up a relation between $\psi_n(\cdot)$ and $\psi_{n-1}(\cdot)$. Realize that, due to the Brownian term, ruin can occur before the exponential clock (with parameter $\vartheta$) expires; this happens with probability $ e^{-\nu_n^+ u}$. 
Following the approach we have been using in the case without the Brownian term, we thus obtain the relation
\[p_n(u) =  e^{-\nu_n^+ u}  + I_n(u,\vartheta),\]
where
\begin{align*}I_n(u,\vartheta)&:=\int_0^u \int_0^\infty\nu_n^+ e^{-\nu_n^+ v}\nu_n^- e^{-\nu_n^- w}
\frac{\lambda_n}{\lambda_n+\vartheta}\,{\mathbb P}(Z_{n-1}+L\geqslant u-v+w)\,{\rm d}w\,{\rm d}v\\
&=\frac{\lambda_n}{\lambda_n+\vartheta}\int_0^u \int_{u-v}^\infty\nu_n^+ e^{-\nu_n^+ v}\nu_n^- e^{-\nu_n^- (z-u+v)}\,
{\mathbb P}(Z_{n-1}+L\geqslant z)\,{\rm d}z\,{\rm d}v
.\end{align*}

The next step is to evaluate $\psi_n(\gamma)$, by multiplying $p_n(u)$ by $e^{-\gamma u}$ and integrating over $u\in[0,\infty).$ We obtain that, interchanging the order of the integrals such that the `easy' integration (over $u$, that is) can be done first,
\begin{align*}\int_0^\infty &e^{-\gamma u} I_n(u,\vartheta){\rm d}u \\&= \frac{\lambda_n}{\lambda_n+\vartheta}
\int_0^\infty\int_0^\infty\int_v^{z+v} e^{-\gamma u} \,\nu_n^+ e^{-\nu_n^+ v}\nu_n^- e^{-\nu_n^- (z-u+v)}
\,{\mathbb P}(Z_{n-1}+L\geqslant z)\,{\rm d}u\,{\rm d}v\,{\rm d}z\\
&=
 \frac{\lambda_n}{\lambda_n+\vartheta}
\int_0^\infty\int_0^\infty {\nu_n^-}e^{-\gamma v}\frac{e^{-\nu_n^- z}-e^{-\gamma z}}{\gamma-\nu_n^-}\nu_n^+e^{-\nu_n^+ v}
\,{\mathbb P}(Z_{n-1}+L\geqslant z)\,{\rm d}v\,{\rm d}z\\
&= \frac{\lambda_n}{\lambda_n+\vartheta} \frac{\nu_n^-\nu_n^+}{(\gamma-\nu_n^-)(\gamma+\nu_n^+)} \int_0^\infty 
\big(e^{-\nu_n^- z}-e^{-\gamma z}\big) \,{\mathbb P}(Z_{n-1}+L\geqslant z)\,{\rm d}z.
\end{align*}
Performing the same steps as in the proof of Theorem \ref{TH1}, as before relying on the identities \eqref{e1} and \eqref{e2} in combination with the independence of $L$ and $Z_{n-1}$, we find
after some standard algebra the following result.

\begin{theorem} \label{TH41} For any $\gamma\geqslant 0$ and $n\in{\mathbb N}$, we have the recursion,
\begin{align*}\psi_n(\gamma)&=\frac{1}{\gamma+\nu_n^+}+\frac{\lambda_n}{\lambda_n+\vartheta}
\frac{1}{\gamma+\nu_n^+}{\frac{\nu_n^+}{\gamma}}\\
&\hspace{20mm}-\frac{\lambda_n}{\lambda_n+\vartheta} \frac{\nu_n^-\nu_n^+}{(\gamma-\nu_n^-)(\gamma+\nu_n^+)}\big( B(\nu_n^-,\psi_{n-1}(\nu_n^-))-B(\gamma,\psi_{n-1}(\gamma))\big),
\end{align*}
where $\psi_0(\gamma)\equiv 0.$
\end{theorem}

\begin{remark}{\em
In Theorem \ref{TH41} we can simplify
\[\frac{\lambda_n}{\lambda_n+\vartheta} \frac{\nu_n^-\nu_n^+}{(\gamma-\nu_n^-)(\gamma+\nu_n^+)}=\frac{\lambda_n}{\lambda_n+\vartheta+r_n\gamma-\frac{1}{2}\gamma^2\sigma_n^2},\]
using that $\nu_n^+$ and $\nu_n^-$ solve
$(\lambda_n+\vartheta)+r_n\alpha-\frac{1}{2}\alpha^2\sigma_n^2=0$. $\hfill\Diamond$
}\end{remark}

\subsection{Multiple groups}\label{MG}

To make the model more realistic, one could work with multiple (heterogeneous) groups of obligors. Suppose there are $G\in\mathbb{N}$ groups of obligors with initially $n_j$ obligors in group $j\in\{1,\ldots,G\}$; write ${\boldsymbol n}=(n_1,\ldots, n_G).$ We consider the multi-group counterpart of the base model of Section \ref{S3}:
each obligor in group $j$ has a time-to-default that is exponentially distributed with rate $\lambda_j$. The losses at default per obligor in group $j$ are i.i.d.\ random variables with Laplace transform $\ell_j(\cdot)$; in addition these per-group sequences are assumed independent. The income per unit time for this group is $r_j i$ when there are $i\in\{1,\ldots, n_j\}$ obligors that have not gone into default yet. 

The company's capital reserve is given by the sum of the reserves of the individual groups; its initial level is $u>0$. Let $\psi_{{\boldsymbol n}}(\gamma)$ denote the double transform of the probability of ruin over an exponentially distributed interval (with, as usual, mean $\vartheta^{-1}$), given there $n^j$ obligors in group $j$ that have not gone into default yet. Then by the same argumentation as before we find, for ${\boldsymbol n}$ component-wise at least equal to 1, and with ${\boldsymbol e}_j$ the $j$-th unit vector,

\begin{align*}
\psi_{\boldsymbol n}(\gamma)=&\sum_{j=1}^G\frac{\lambda_j n_j}{\sum_{k=1}^G\lambda_k n_k+\vartheta}\frac{1}{\gamma}+\sum_{j=1}^G\frac{\lambda_j n_j}{\sum_{k=1}^G\lambda_k n_k +\vartheta-\gamma r_j n_j}\\
&\times\Bigg(B_j\left(\frac{\lambda_j +\vartheta/n_j}{r_j},\psi_{{\boldsymbol n}-{\boldsymbol e}_j}\left(\frac{\lambda_j +\vartheta/n_j}{r_j}\right)\right)-B_j\left(\gamma,\psi_{{\boldsymbol n}-{\boldsymbol e}_j}(\gamma\right)\Bigg),
\end{align*}
\normalsize
where \[B_j(x,y):= \ell_j(x)\left(\frac{1}{x}-y\right).\] We have thus expressed $\psi_{\boldsymbol n}(\gamma)$ as a linear function of $\psi_{{\boldsymbol n}-{\boldsymbol e}_1}(\gamma)$ up to $\psi_{{\boldsymbol n}-{\boldsymbol e}_G}(\gamma)$.
A similar recursive relation be found if some of the entries of ${\boldsymbol n}$ equal 0. Given that $\psi_{{\boldsymbol 0}}(\gamma)=0$, with ${\boldsymbol 0}$ denoting the $G$-dimensional all-zeroes vector, we have thus devised a procedure to identify $\psi_{\boldsymbol n}(\gamma)$.

\begin{remark}{\em 
The above model extension with multiple classes offers an important additional flexibility. In the first place, one could cluster the obligors in terms of the loss distributions. Per class this loss can even be deterministic; this is a useful property, as in the credit context the losses of some obligors may be a priori known. In addition, we could work with some classes in which the obligors do not go bankrupt and some classes in which they do. Also, one could work with a class-specific income rate.  $\hfill\Diamond$
}\end{remark}

\section{Numerical experiments}\label{num}
In this section we focus on issues concerning the numerical evaluation of the ruin probability. 
In the first subsection, we specialize to the case that the losses are exponentially distributed, where some of the quantities that feature in the numerical analysis allow closed-form analysis. In the second subsection, we present a couple of illustrative examples. These in particular quantify the effect of the size of the obligor population.

\subsection{Exponentially distributed losses}\label{subsec_exp}
In Section \ref{subsec_analysis} the focus was on finding an expression for the double transform $\psi_n(\gamma)$,
which can then be inverted numerically. In Section~\ref{asy} we presented a couple of other approaches: asymptotics, an efficient importance sampling algorithm, and bounds. In this section we present an alternative technique, namely an iterative procedure that directly provides the ruin probabilities $p_n(u,t)$ themselves.  We consider the model variant in which the default rate and the income rate are $\lambda_i$ and $r_i$, respectively, during time periods in which there are $i$ obligors left. 

As in Section \ref{subsec_analysis}, the idea is to condition on the first default.
We thus obtain, with $W(\cdot)$ as introduced in Section \ref{asy}, the following recursive relation:
\begin{align}\nonumber
p_n(u,t)&=\int_0^t  \lambda_n e^{- \lambda_n s} {\mathbb P}\left(\sup_{0\leqslant v\leqslant t-s} \sum_{i=1}^{n-1}W_i(v)+L\geqslant u+r_ns\right)\,{\rm d}s\\ \nonumber
&=\int_0^t \lambda_n e^{- \lambda_ns} \,{\rm d}s-\int_0^t  \lambda_n e^{- \lambda_n s} {\mathbb P}\left(\sup_{0\leqslant v\leqslant t-s} \sum_{i=1}^{n-1}W_i(v)+L\leqslant u+r_ns\right)\,{\rm d}s\\
&=1-e^{-\lambda_n t}-\int_0^t\int_0^{u+r_n s}  \lambda_n e^{- \lambda_n s}\left(1-p_{n-1}(u+r_n s-x,t-s)\right)\mathbb{P}(L\in {\rm d}x)\,{\rm d}s. \label{ECURS}
\end{align}

When there is only one obligor left, there is only one scenario leading to ruin: default should take place before the exponential clock (with mean $\vartheta^{-1}$)  expires and the loss should be sufficiently large. In other words, 
\begin{align}
\nonumber p_1(u,t)&= \int_0^t\int^\infty_{u+r_1s}  \lambda_1 e^{- \lambda_1 s} \mathbb{P}(L\in {\rm d}x)\,{\rm d}s= \int_0^t \lambda_1 e^{- \lambda_1 s} \mathbb{P}(L\geqslant u+r_1s)\,{\rm d}s
\end{align}
From this point on we focus on the case of exponentially distributed claims with mean $\mu^{-1}$, i.e., ${\mathbb P}(L\geqslant x)=e^{-\mu x}$. We readily obtain
\[p_1(u,t)=\int_0^t \lambda_1 e^{- \lambda_1 s} e^{-\mu (u+r_1s)}\,{\rm d}s=\frac{\lambda_1 e^{-\mu u}}{\lambda_1+\mu r_1}\left(1-e^{-(\lambda_1+\mu r_1) t}\right).\]
We can thus obtain $p_2(u,t)$ applying numerical integration to \eqref{ECURS} with $n=2$. Continuing along these lines,  $p_n(u,t)$ can be numerically evaluated for higher values of $n$. 

\vb

We now point out how to evaluate the large-deviations asymptotics that were presented in Section~\ref{subsec_LD}, in the case of exponentially distributed claims.
The moment generating function of $W_1(s)$ is for $\alpha<\mu$ given by
\[\omega_s(\alpha)=\left(1-e^{-(\lambda+ r\alpha)s}\right) \frac{\lambda}{\lambda+r\alpha}\frac{\mu}{\mu-\alpha} + e^{-(\lambda+r\alpha)s},\]
whereas for $\alpha\geqslant \mu$ the moment generating function is infinite. We continue by computing the mean net loss corresponding to a single obligor (as a function of time):
\begin{align*}m(s)&:={\mathbb E}W_1(s) = \frac{1}{\mu}(1-e^{-\lambda s}) - r\int_0^s u\, \lambda e^{-\lambda v}{\rm d}v
- rs\int_s^\infty  \lambda e^{-\lambda v}{\rm d}v\\
&=\left(\frac{1}{\mu}-\frac{r}{\lambda}\right)(1-e^{-\lambda s}).
\end{align*}
{In the sequel we will assume $u>m(\infty)$, or equivalently $\lambda-r\mu<\lambda\mu u$, to make sure the event under consideration is rare.}

The Legendre transform pertaining to $W_1(s)$ reads
\[I(s):=\sup_{0<\alpha<\mu}\left(\alpha u - \log \omega_s(\alpha)\right);\]
we can rule out $\alpha\geqslant \mu$ as  $\omega_s(\alpha)=\infty$ for these $\alpha$.
Because the first-order condition does not allow an explicit solution, one cannot write $I(s)$ in closed form. 
Two boundary cases can be dealt with explicitly, though. It is first observed that, denoting by $\omega'_{s,1}(\alpha)$ the derivative of $\omega_s(\alpha)$ with respect to $\alpha$, and by $\omega'_{s,2}(\alpha)$ the derivative of $\omega_s(\alpha)$ with respect to $s$,
\begin{align}\nonumber I'(s)& = \frac{\rm d}{{\rm d}s} \left(\alpha^\star(s) u - \log \omega_s(\alpha^\star(s))\right)\\&=\frac{{\rm d}\alpha^\star(s)}{{\rm d}s}\left(u- \frac{\omega'_{s,1}(\alpha^\star(s))}{\omega_s(\alpha^\star(s))}\right)- \frac{\omega'_{s,2}(\alpha^\star(s))}{\omega_s(\alpha^\star(s))}=- \frac{\omega'_{s,2}(\alpha^\star(s))}{\omega_s(\alpha^\star(s))},\label{IS}\end{align}
where the last equality is due to the definition of $\alpha^\star(s)$. By an elementary computation,
\begin{equation}\label{afge}\omega'_{s,2}(\alpha)= \left(\frac{\lambda\mu}{\mu-\alpha}-(\lambda+r\alpha)\right)e^{-(\lambda +r\alpha)s}=
\frac{r\alpha^2+\lambda\alpha -r\mu\alpha}{\mu-\alpha}\,e^{-(\lambda +r\alpha)s}.\end{equation}
We observe that the Legendre transform $I(s)$ is decreasing in $s$ whenever $\alpha^*(s)>\mu-{\lambda}/{r}$.
\begin{itemize}
\item[$\circ$] For $s=0$, we immediately see that $\omega_0(\alpha) = 1$ for all $\alpha$, so that $\alpha^\star(0)=\mu$ and $I(0) = \mu u.$ In addition, we obtain by some straightforward algebra that
\[I'(0) = - \lim_{\alpha\uparrow \mu} \frac{\omega'_{0,2}(\alpha)}{\omega_0(\alpha)}=-\infty.\]
\item[$\circ$] For $s=\infty$,
\[I(s) = \sup_{0<\alpha<\mu} \kappa(\alpha),\:\:\:\:\kappa(\alpha):=\alpha u - \log (\lambda\mu) +\log(\lambda+r\alpha)+\log(\mu-\alpha).\]
Observe that $\kappa(\cdot)$ is concave, with $\kappa'(0)>0$ (under the assumption $u>m(\infty)$) and $\kappa(\alpha)\to-\infty$ as $\alpha\uparrow\mu.$ In other words, $\kappa(\cdot)$ attains a maximum in $(0,\mu).$
The first order condition, determining $\alpha^\star(\infty)$, is
\[u=\frac{1}{\mu-\alpha}-\frac{r}{\lambda+r\alpha},\]
or equivalently
\[ru\alpha^2 +\big((\lambda-r\mu )u+2r\big)\alpha - \lambda\mu\big(u-m(\infty)\big)=0.\]
As $\lambda\mu (u-m(\infty))>0$, this equation has a positive and negative root. 
Consequently, $\alpha^\star(\infty)$ is the positive root, i.e.,
\[\alpha^\star(\infty) = \frac{-2r-\lambda u+r\mu  u +\sqrt{4r^2+\lambda^2 u^2+2r\lambda\mu  u^2 + r^2\mu^2  u^2}}{2ru},\]
so that $I(\infty) = \kappa(\alpha^\star(\infty)).$
Next, we want to find the sign  of $I(s)$ in the regime that $s\to\infty$. Based on \eqref{IS} and \eqref{afge}, this is the sign of 
$-r\alpha^\star(\infty)-\lambda+r\mu$. Using the explicit solution of $\alpha^\star(\infty)$, it requires some straightforward calculus to verify that this leads to a negative sign, i.e. $I(s)$ is decreasing in the regime that $s\to\infty$, if and only if $ \lambda- r\mu>-\lambda\mu u$. 
\end{itemize}

\subsection{Numerical example}
For the numerical results we have used a setup that aligns with the one considered in \cite{Asmussen1984}. 
\begin{enumerate}
\item[$\circ$] We consider the case that both the income rates $r_i$ and the default intensity $\lambda_i$ are linear in the number of obligors $i$ that have not gone into default yet. We let the proportionality constants be $r=  1$ and $\lambda=0.9$, respectively. In other words, when there are $i$ obligors in the system that have not gone into default yet, the income rate is given by $i$ and the default intensity rate by $0.9\,i$.
\item[$\circ$] The losses are exponentially distributed with parameter $\mu=1$.
\end{enumerate}
With these parameter settings the rarity condition $m(\infty)<u$ is satisfied for all $u>0$, as we have that $0.9-1=-0.1<0<0.9\,u$.

First, we focus on the evaluation of the large-deviation asymptotics. For $s\rightarrow \infty$ we have that the Legendre transform $I(s)$ is decreasing (increasing) if $u>\frac{1}{9}$ (if $u<\frac{1}{9}$, respectively). For illustrational purposes we have plotted the functions $\alpha^\star(s)$ and $I(s)$ in Figure \ref{fig1}, as a function of time $s$, for $u=5$ as well as $u=0.1$.
 In the first instance, with $u=5$, the function $I(\cdot)$ is decreasing, so that the optimal $t^\star=\infty$, whereas for $u=0.1$ we see that $I(\cdot)$ attains a minimal value at $t^\star=2.3$. 

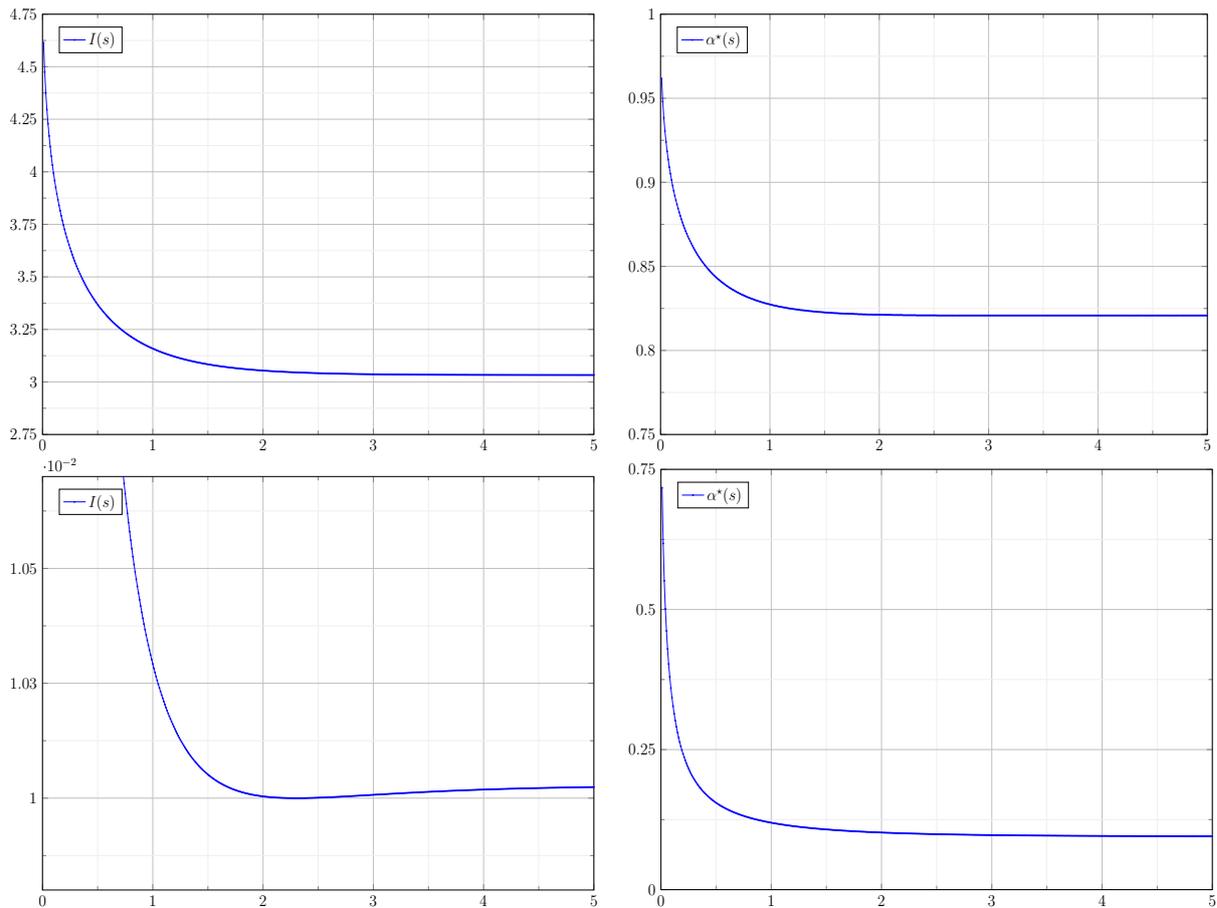
\begin{figure}
  \centering
\resizebox{8cm}{6cm}{
 \pgfplotstableread{Data.txt}{\table}
    \begin{tikzpicture}
        \begin{axis}[
            xmin = 0, xmax = 5,
            ymin = 2.75, ymax = 4.75,
            xtick distance = 1,
            ytick distance = 0.25,
            grid = both,
            minor tick num = 1,
            major grid style = {lightgray},
            minor grid style = {lightgray!25},
            width = \textwidth,
            height = 0.75\textwidth,
            legend cell align = {left},
            legend pos = north west
        ]
            \addplot[blue, mark = *, mark size = 0.3pt] table [x = {x}, y = {b}] {\table};
                       \legend{$I(s)$}
        \end{axis}
    \end{tikzpicture}}
  \resizebox{8cm}{6cm}{
 \pgfplotstableread{Data.txt}{\table}
    \begin{tikzpicture}
        \begin{axis}[
            xmin = 0, xmax = 5,
            ymin = 0.75, ymax = 1,
            xtick distance = 1,
            ytick distance = 0.05,
            grid = both,
            minor tick num = 1,
            major grid style = {lightgray},
            minor grid style = {lightgray!25},
            width = \textwidth,
            height = 0.75\textwidth,
            legend cell align = {left},
            legend pos = north west
        ]
            \addplot[blue, mark = *,mark size = 0.3pt] table [x = {x}, y = {a}] {\table};
         
                     \legend{
                $\alpha^\star(s)$}
        \end{axis}
    \end{tikzpicture}
}


\resizebox{8cm}{6cm}{
 \pgfplotstableread{Data.txt}{\table}
    \begin{tikzpicture}
        \begin{axis}[
            xmin = 0, xmax = 5,
            ymin = 0.0098, ymax = 0.0107,
            xtick distance = 1,
            ytick distance = 0.00025,
            grid = both,
            minor tick num = 1,
            major grid style = {lightgray},
            minor grid style = {lightgray!25},
            width = \textwidth,
            height = 0.75\textwidth,
            legend cell align = {left},
            legend pos = north west
        ]
            \addplot[blue, mark = *, mark size = 0.3pt] table [x = {x}, y = {d}] {\table};
          
            \legend{$I(s)$}
        \end{axis}
    \end{tikzpicture}}
    \resizebox{8cm}{6cm}{
 \pgfplotstableread{Data.txt}{\table}
    \begin{tikzpicture}
        \begin{axis}[
            xmin = 0, xmax = 5,
            ymin = 0, ymax = 0.75,
            xtick distance = 1,
            ytick distance = 0.25,
            grid = both,
            minor tick num = 1,
            major grid style = {lightgray},
            minor grid style = {lightgray!25},
            width = \textwidth,
            height = 0.75\textwidth,
            legend cell align = {left},
            legend pos = north west
        ]
            \addplot[blue, mark = *, mark size = 0.3pt] table [x = {x}, y = {c}] {\table};
          
            \legend{
                $\alpha^\star(s)$}
        \end{axis}
    \end{tikzpicture}}

  \caption{The Legendre transform $I(s)$ and the underlying optimal $\alpha^\star(s)$ parameter as a function of time $s$ (for $s\in[0,5]$). In the top panels we took  for $u=5$, whereas in the bottom panels we took $u=0.1$.}\label{fig1}
\end{figure}

\vb

In Figure \ref{fig_Exact_n10} we present, for different values of the initial number of obligors $n$ and $u=5$,  the ruin probabilities as a function of time. This has been done relying on the iterative approach presented of Section \ref{subsec_exp}. The double integral involved has been evaluated analytically for $n=1,2$ while numerical integration methods have been employed for $n>2$. We do observe that the ruin probability increases in the length of the time interval, as desired. 
The upper bound (as derived in Section \ref{Unif}) in this instance is given by 0.6065, and is independent of the number of obligors $n$. As can be observed, this upper bound is rather conservative, in particular  when there are only a few obligors in the system.

\begin{figure}\resizebox{12.8cm}{9cm}{
 \pgfplotstableread{Data2.txt}{\table}\pgfplotsset{scaled y ticks=false}
    \begin{tikzpicture}
        \begin{axis}[
            xmin = 0, xmax = 10,
            ymin = 0, ymax = 0.15,
            xtick distance = 1,
            ytick distance = 0.01,
            grid = both,
            minor tick num = 1,
            major grid style = {lightgray},
            minor grid style = {lightgray!25},
            width = \textwidth,
            height = 0.75\textwidth,
            legend cell align = {left},
            yticklabel style={
        /pgf/number format/fixed,
        /pgf/number format/precision=5
},
scaled y ticks=false,
            legend pos = north west]
                       \addplot[blue, mark = *, mark size = 0.3pt] table [x = {x}, y = {a}] {\table};
             \addplot[blue, mark = *, mark size = 0.3pt] table [x = {x}, y = {b}] {\table};
              \addplot[blue, mark = *, mark size = 0.3pt] table [x = {x}, y = {c}] {\table};
               \addplot[blue, mark = *, mark size = 0.3pt] table [x = {x}, y = {d}] {\table};
                \addplot[blue, mark = *, mark size = 0.3pt] table [x = {x}, y = {e}] {\table};
                 \addplot[blue, mark = *, mark size = 0.3pt] table [x = {x}, y = {f}] {\table};
                  \addplot[blue, mark = *, mark size = 0.3pt] table [x = {x}, y = {g}] {\table};
                   \addplot[blue, mark = *, mark size = 0.3pt] table [x = {x}, y = {h}] {\table};
                    \addplot[blue, mark = *, mark size = 0.3pt] table [x = {x}, y = {i}] {\table};
                     \addplot[blue, mark = *, mark size = 0.3pt] table [x = {x}, y = {j}] {\table};

        \end{axis}
         \end{tikzpicture}}
\caption{Ruin probabilities over time: $p_n(u,t)$ as a function of $t$, for $n=1$ (bottom line) to $n=10$ (top line), with $u=5$.}
\label{fig_Exact_n10}
\end{figure}
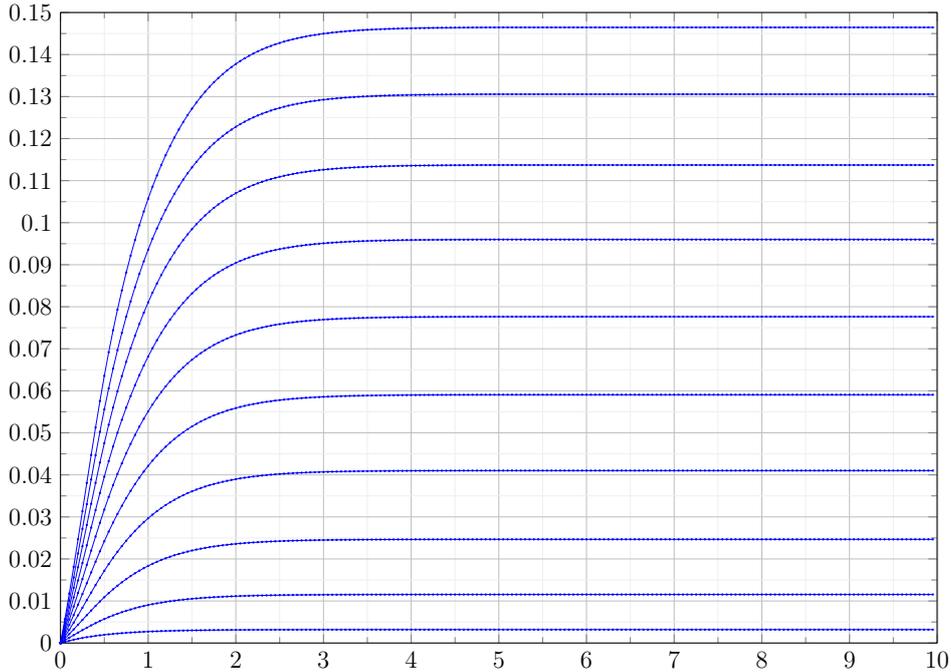

In a next experiment we study the performance of the importance sampling technique that was presented Section \ref{subsec_is}. The top panel of Figure \ref{fig3}  shows, for the initial capital reserve $u$ being equal to $5$, the estimates of the ruin probability as a function of time, obtained by 
simulation, using our importance sampling algorithm. The values nearly coincide with what is obtained 
by applying the na\"{i}ve, direct simulation approach (i.e., without a change of measure); 
from Figure \ref{fig_Exact_n10} we in addition observe that there is a highly accurate match with the values computed using the iterative approach of Section \ref{subsec_exp}. 
Regarding the  importance sampling simulations  it is noted that we let the events ${\mathscr E}_j$ correspond to the event where the net cumulative loss process exceeds the initial level $u$ (instead of $nu$), as $u$ in this example corresponds to the {\it unscaled} initial capital level. The fact that we have used as many as $10^6$ runs guarantees estimates with a high precision. The importance sampling based approach substantially outperforms direct simulation, in that it greatly reduces the variance of the estimator, as can be observed in the bottom panels of Figure \ref{fig3}.

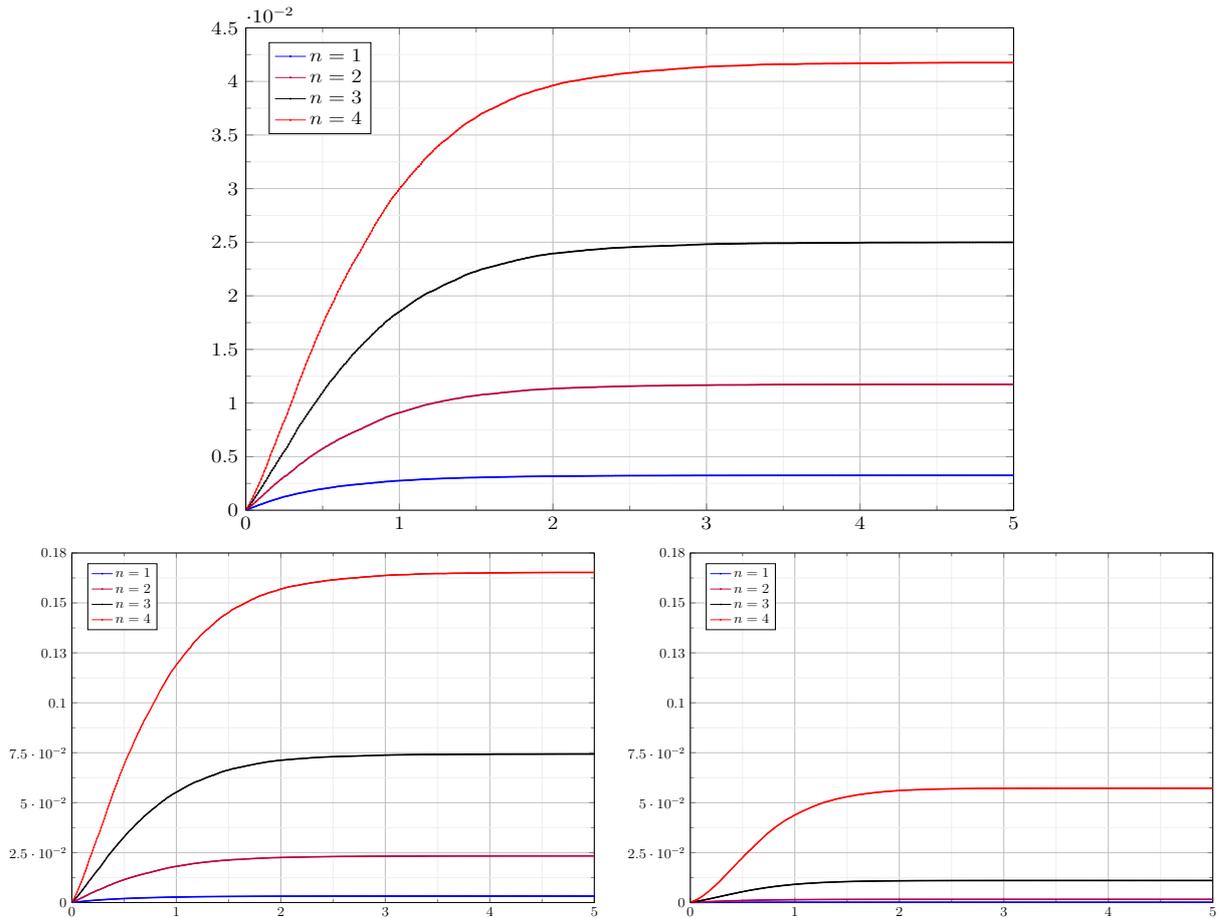
\begin{figure}
  \centering
\begin{center}
\resizebox{11cm}{7cm}{
     \pgfplotstableread{Data3.txt}{\table}
    \begin{tikzpicture}    
        \begin{axis}[
            xmin = 0, xmax = 5,
            ymin = 0, ymax = 0.045,
            xtick distance = 1,
            ytick distance = 0.005,
            grid = both,
            minor tick num = 1,
            major grid style = {lightgray},
            minor grid style = {lightgray!25},
            width = \textwidth,
            height = 0.75\textwidth,
            legend cell align = {left},
            legend pos = north west
        ]
            \addplot[blue, mark = *, mark size = 0.3pt] table [x = {x}, y = {regular simulation n=1}] {\table};
             \addplot[purple, mark = *, mark size = 0.3pt] table [x = {x}, y = {regular simulation n=2}] {\table};
              \addplot[black, mark = *, mark size = 0.3pt] table [x = {x}, y = {regular simulation n=3}] {\table};
               \addplot[red, mark = *, mark size = 0.3pt] table [x = {x}, y = {regular simulation n=4}] {\table};
                      
                       \legend{
                $n=1$,
                $n=2$,
                $n=3$,
                $n=4$}   
           
        \end{axis}
         \end{tikzpicture}}\end{center}

\resizebox{8cm}{4.99cm}{
     \pgfplotstableread{Data3.txt}{\table}
    \begin{tikzpicture}
        \begin{axis}[
            xmin = 0, xmax = 5,
            ymin = 0, ymax = 0.175,
            xtick distance = 1,
            ytick distance = 0.025,
            grid = both,
            minor tick num = 1,
            major grid style = {lightgray},
            minor grid style = {lightgray!25},
            width = \textwidth,
            height = 0.75\textwidth,
            legend cell align = {left},
            legend pos = north west
        ]
            \addplot[blue, mark = *, mark size = 0.3pt] table [x = {x}, y = {regular variance n=1}] {\table};
             \addplot[purple, mark = *, mark size = 0.3pt] table [x = {x}, y = {regular variance n=2}] {\table};
              \addplot[black, mark = *, mark size = 0.3pt] table [x = {x}, y = {regular variance n=3}] {\table};
               \addplot[red, mark = *, mark size = 0.3pt] table [x = {x}, y = {regular variance n=4}] {\table};
                      
                       \legend{
                $n=1$,
                $n=2$,
                $n=3$,
                $n=4$}   
           
        \end{axis}
         \end{tikzpicture}}
         \resizebox{8cm}{4.99cm}{
     \pgfplotstableread{Data3.txt}{\table}
    \begin{tikzpicture}
        \begin{axis}[
            xmin = 0, xmax = 5,
            ymin = 0, ymax = 0.175,
            xtick distance = 1,
            ytick distance = 0.025,
            grid = both,
            minor tick num = 1,
            major grid style = {lightgray},
            minor grid style = {lightgray!25},
            width = \textwidth,
            height = 0.75\textwidth,
            legend cell align = {left},
            legend pos = north west
        ]
                \addplot[blue, mark = *, mark size = 0.3pt] table [x = {x}, y = {is variance n=1}] {\table};
             \addplot[purple, mark = *, mark size = 0.3pt] table [x = {x}, y = {is variance n=2}] {\table};
              \addplot[black, mark = *, mark size = 0.3pt] table [x = {x}, y = {is variance n=3}] {\table};
               \addplot[red, mark = *, mark size = 0.3pt] table [x = {x}, y = {is variance n=4}] {\table};
                      
                       \legend{
                $n=1$,
                $n=2$,
                $n=3$,
                $n=4$}   
           
        \end{axis}
         \end{tikzpicture}}

\caption{Top panel: ruin probabilities, as simulated by importance sampling: $p_n(u,t)$ as a function of time $t$. Bottom left panel: variance of the estimator under direct simulation as a function of $t$. Bottom right panel: variance of the estimator under importance sampling as a function of $t$. In all experiments we took $u=5.$}\label{fig3}
\end{figure}

\section{Concluding remarks}
Motivated by applications in credit risk,  we have analyzed in this paper a transient counterpart of the classical Cram\'er-Lundberg model. We have presented a broad range of results: exact analysis in terms of transforms, asymptotic analysis including an efficient rare-event simulation algorithm, and four model variants (viz.\ a setup that also includes non-default losses, one with Markov modulation to make the obligors dependent, one in which the linear drifts are replaced by Brownian motions, and a last one in which there are multiple groups of obligors). 

Follow-up research could relate to the next steps to make this model operational. A main challenge concerns dealing with the heterogeneity between the obligors. When there are relatively few groups (with homogeneity within these groups) the approach of Section \ref{MG} can be relied upon, but when effectively all obligors have a specific time-to-default and loss distribution, an alternative approach needs to be developed. Another topic for future research could concern procedures to on-the-fly adjust the capital level given realizations of the defaults; cf.\ e.g.\ the approach proposed in\cite{DMSW}. 


\bibliographystyle{plain}
\end{document}